\documentclass[12pt]{amsart}
\usepackage[utf8]{inputenc}
\usepackage{mathtools,hyperref,quiver}
\usepackage[maxnames=8]{biblatex}
\usepackage[capitalize]{cleveref}
\usepackage[capitalize]{cleveref}
\usepackage[margin=2cm]{geometry}

\usepackage{amsthm}
\usepackage[thmtools-compat]{keytheorems}

\newkeytheorem{thm}[parent=section]

\NewDocumentCommand{\DefThm}{m m m m}{
  \newkeytheorem{#2}[
    style={#1},
    sibling=thm,
    title={#3},
    refname={{#3},{#4}},
  ]
}

\NewDocumentCommand{\DefTheorem}{m m m}{\DefThm{plain}{#1}{#2}{#3}}
\NewDocumentCommand{\DefNonTheorem}{m m m}{\DefThm{definition}{#1}{#2}{#3}}

\DefTheorem{theorem}{Theorem}{Theorems}
\DefTheorem{lemma}{Lemma}{Lemmas}
\DefTheorem{corollary}{Corollary}{Corollaries}
\DefTheorem{proposition}{Proposition}{Propositions}
\DefTheorem{conjecture}{Conjecture}{Conjectures}

\DefNonTheorem{definition}{Definition}{Definitions}
\DefNonTheorem{construction}{Construction}{Constructions}
\DefNonTheorem{example}{Example}{Examples}
\DefNonTheorem{remark}{Remark}{Remarks}

\usepackage{enumitem}

\newlist{conditions}{enumerate}{1}
\setlist[conditions]{label=(\arabic*),itemsep=0ex}
\Crefname{conditionsi}{Condition}{Conditions}
\crefname{conditionsi}{condition}{conditions}

\newlist{parts}{enumerate}{1}
\setlist[parts]{label=(\arabic*),itemsep=0ex}
\Crefname{partsi}{Part}{Parts}
\crefname{partsi}{part}{parts}

\newcounter{nodemaker}
\def\twocellflexible#1#2#3{%
  \global\edef\mynodeone{twocell\arabic{nodemaker}}%
  \stepcounter{nodemaker}%
  \global\edef\mynodetwo{twocell\arabic{nodemaker}}%
  \stepcounter{nodemaker}%
  \ar[#1,phantom, start anchor=center, end anchor=center, shift left={#2}, "" {name=\mynodeone}, description]%
  \ar[#1,phantom, start anchor=center, end anchor=center, shift right={#2}, "" {name=\mynodetwo}, description]%
  \ar[Rightarrow,anchor=center,from=\mynodeone,to=\mynodetwo,#3]%
}

\NewDocumentCommand{\twocell}{m O{}}{\twocellflexible{#1}{4}{#2}}
\NewDocumentCommand{\twocellflip}{m O{}}{\twocellflexible{#1}{-4}{#2}}

\DeclarePairedDelimiter\bracks\lbrack\rbrack
\DeclarePairedDelimiter\verts\lvert\rvert

\DeclarePairedDelimiter\braces\lbrace\rbrace
\DeclarePairedDelimiterX\set[2]\lbrace\rbrace{#1 \mathrel{\delimsize\vert} #2}

\usepackage{xspace}

\makeatletter
\newcommand{\DeclareAbbrevation}[2]{\newcommand{#1}{\@ifnextchar{.}{#2}{#2.\@\xspace}}}
\makeatother

\DeclareAbbrevation{\ie}{i.e}
\DeclareAbbrevation{\eg}{e.g}
\DeclareAbbrevation{\cf}{cf}
\DeclareAbbrevation{\etc}{etc}
\DeclareAbbrevation{\resp}{resp}
\DeclareAbbrevation{\etal}{et al}
\DeclareAbbrevation{\ibid}{ibid}
\DeclareAbbrevation{\ca}{ca}

\DeclareMathOperator{\colim}{\mathrm{colim}}

\newcommand{\dom}{\mathsf{dom}}
\newcommand{\cod}{\mathsf{cod}}
\newcommand{\fst}{\mathsf{fst}}

\newcommand{\id}{\mathsf{id}}
\newcommand{\II}{\mathbb{I}}
\newcommand{\nat}{\mathbb{N}}
\newcommand{\cat}{\mathsf{Cat}}
\newcommand{\spc}{\mathcal{S}}

\newcommand{\arr}{\mathrm{Ar}}

\newcommand{\cocart}{\mathsf{cocart}}

\newcommand{\arrcocart}{\arr_\cocart}

\newcommand{\lfib}{\mathrm{LFib}}
\newcommand{\rfib}{\mathrm{RFib}}

\newcommand{\cocartfib}{\mathrm{CocartFib}}

\newcommand{\fun}{\mathrm{Fun}}

\newcommand{\funcocart}{\mathrm{Fun}_\cocart}

\newcommand{\gl}{\mathrm{Gl}}
\newcommand{\map}{\mathrm{Map}}
\newcommand{\fib}{\mathrm{fib}}
\newcommand{\ob}{\mathrm{Ob}}
\newcommand{\ev}{\mathsf{ev}}
\newcommand{\co}{\colon}

\newcommand{\cocomma}[2]{{#1} \uparrow {#2}}

\makeatletter
\def\slashedarrowfill@#1#2#3#4#5{%
  $\m@th\thickmuskip0mu\medmuskip\thickmuskip\thinmuskip\thickmuskip
  \relax#5#1\mkern-7mu%
  \cleaders\hbox{$#5\mkern-2mu#2\mkern-2mu$}\hfill
  \mathclap{#3}\mathclap{#2}%
  \cleaders\hbox{$#5\mkern-2mu#2\mkern-2mu$}\hfill
  \mkern-7mu#4$%
}
\def\rightslashedarrowfill@{%
  \slashedarrowfill@\relbar\relbar\mapstochar\rightarrow}
\newcommand\xslashedrightarrow[2][]{%
  \ext@arrow 0055{\rightslashedarrowfill@}{#1}{#2}}
\makeatother

\newcommand{\profto}{\xslashedrightarrow{}}

\DeclareFontFamily{U}{min}{}
\DeclareFontShape{U}{min}{m}{n}{<-> udmj30}{}
\newcommand\yo{\!\text{\usefont{U}{min}{m}{n}\symbol{'207}}\!}
\newcommand\yogl{\yo^{\gl}}

\title{A synthetic construction of universal cocartesian fibrations}
\author{Christian Sattler}
\author{David Wärn}

\begin{document}

\begin{abstract}
We give a model-independent construction of directed univalent cocartesian
fibrations of $(\infty,1)$-categories, and prove a straightening
equivalence against such fibrations.
The key step is showing that cocartesian fibrations descend along
localisations, which we accomplish by analysing mapping spaces of localisations.
Along the way we introduce a directed version of the join construction, giving a
sequential colimit description of the full image of any functor.
\end{abstract}

\maketitle

\section{Introduction}

In ordinary category theory, one often gets away with constructing categories and functors between them in an ad-hoc manner, by declaring what objects and morphisms are, and how functors act on them.
In higher category theory however, this would involve an infinite tower of coherences and quickly becomes infeasible. A more principled approach is needed.
In Lurie's foundational work, this approach makes use of set-based models of homotopical structures, particularly quasicategories.
In contrast, more recent work on the foundations and applications of higher category theory increasingly uses the resulting high-level language of $(\infty, 1)$-categories, which is model-independent and homotopy-invariant.

In this approach, which one might call \emph{formal} or \emph{synthetic}, ($\infty$-categorical) universal properties play a central role.
For example: in ordinary category theory, one can define the functor category $\fun(A, B)$ by declaring that its objects are functors and its morphisms natural transformations.
In a synthetic setting, we forget about the definition and instead rely on the universal property, which in this case says that functors $X \to \fun(A, B)$ correspond to functors $X \times A \to B$.

Of fundamental importance are the (large) $(\infty, 1)$-categories $\spc$ of $\infty$-groupoids and $\cat_\infty$ of $(\infty, 1)$-categories, and these too enjoy universal properties which go by the name of \emph{straightening--unstraightening}.
Namely, $\spc$ is the target of a \emph{left fibration} $p \co \spc_{\bullet} \to \spc$.
Every functor $F \co X \to \spc$ thus determines a left fibration over $X$, given by the pullback $F^* p \co X \times_\spc \spc_\bullet \to X$; this is the \emph{unstraightening of $F$}.
The universal property of $\spc$ says that \emph{every} left fibration over $X$ with small fibres arises in this way from a functor $F \co X \to \spc$ (its straightening) which is unique up to contractible choice.
Similarly, $\cat_\infty$ is the target of a \emph{cocartesian} fibration $q \co \cat_{\infty \bullet} \to \cat_\infty$, and its universal property expresses that functors $X \to \cat_\infty$ correspond to cocartesian fibrations over $X$ with small fibres, under the same operation of pulling back $q$ along a functor $F \co X \to \cat_\infty$.%
\footnote{%
In the setting of ordinary categories, (co)cartesian fibrations are known as \emph{Grothendieck (op)fibrations}, and the operation of unstraightening a functor $X \to \cat$ is known as the \emph{Grothendieck construction}.
}
For short, we may say that $\spc_\bullet \to \spc$ is the \emph{universal left fibration}, and $\cat_{\infty \bullet} \to \cat_\infty$ is the universal cocartesian fibration.

The existence of a universal cocartesian fibration is a foundational technical result of higher category theory.
Since Lurie's first proof~\cite{lurie09}, simpler treatments have been given~\cite{cisinski22,hebestreit25}, and the result has been generalised to categories internal to $\infty$-topoi~\cite{martini22} and proven in the framework of simplicial type theory~\cite{gratzer26}.
But the question remains whether a model-independent proof is possible.
On the surface this seems unlikely: how could one build the $(\infty, 1)$-category of $(\infty, 1)$-categories without knowing precisely what an $(\infty,1)$-category is?

In the present work, we present a construction of $\cat_\infty$ and prove its universal property using model-independent means, assuming little more than straightening--unstraightening for left fibrations.
The resulting argument is simultaneously simple and general: we expect it to apply also in the setting of categories internal to an $\infty$-topos and in constructive settings.
This can be viewed as a continuation of the program of synthetic homotopy theory; indeed we use ideas from homotopy type theory.

\subsection{Overview}

For a cocartesian fibration $p \co E \to B$, every morphism $f : B(x, y)$ in its base induces a \emph{transport} functor $f_! \co E(x) \to E(y)$ on fibres.
We say that $p$ is \emph{directed univalent} if the corresponding map $B(x, y) \to \map(E(x), E(y))$ is an equivalence of
spaces.
We say that a category%
\footnote{From now on we take \emph{category} to mean $(\infty,1)$-category.}
$F$ is \emph{classified} by $p$ if there is an equivalence $F \simeq E(x)$ for some $x : B$.

Our first main result is that there are \emph{enough} directed univalent cocartesian fibrations, in the following sense.

\getkeytheorem{enough}

This is a refinement of the naive idea that there is a `category of \emph{all} categories'.
It expresses that for any collection of categories which is small enough that it can be indexed by a category, there is a category of \emph{those} categories.

Our second main result is a straightening theorem against any directed univalent cocartesian fibration.

\getkeytheorem{straightening}

The second claim above implies that the straightening $f \co C \to B$ of $q$ is unique in an appropriate sense.
The map it refers to is defined in \cref{cocartesian-transport}.

The key ingredient for both theorems is a proof that cocartesian fibrations descend along cocommas, sequential colimits, and localisations (\cf \cref{descent-cocomma,descent-sequential,descent-localisation}).
Localisations are the hardest to deal with by far.
What we show is that if $i \co C \to D$ exhibits $D$ as the localisation of $C$ at some collection $W \subseteq \arr(C)$ of morphisms, then pulling back along $i$ induces a fully faithful functor
\[
\begin{tikzcd}
  \cocartfib(D)
  \ar[r, hook, "i^*"]
&
  \cocartfib(C)
\end{tikzcd}
\]
whose image consists of those cocartesian fibrations $p \co E \to C$ with the property that for any morphism $f \in W$, the transport functor $f_!$ is invertible.
Luckily, it is not so hard to say what the inverse to $i^*$ at such a cocartesian fibration $p \co E \to C$ ought to be: it ought to be the localisation $E[V^{-1}]$ at the collection $V$ of morphisms that are cocartesian lifts of morphisms belonging to $W$.
Most of the work goes in to showing that $E[V^{-1}] \to D$ is a cocartesian fibration and that its pullback along $i$ is $p \co E \to C$ (\cf \cref{q-prime-cocart}).
This reduces to fairly concrete claims about mapping spaces of $E[V^{-1}]$ (\cref{mapping-spaces-cart}), which we prove using a general, functorial description of mapping spaces of localisations (\cf \cref{ex-localisation}).

The other ingredient is a model-independent way of decomposing an arbitrary category into simpler pieces.
Specifically, given an arbitrary functor $f \co X \to Y$, we explain how to build the full image of $f$ as the colimit of a sequence $X \to X_1 \to X_2 \cdots$ where each $X_n$ is a localisation of a cocomma involving $X$ and $X_{n-1}$ (\cf \cref{join-correctness}).
This is a directed analogue of Rijke's join construction~\cite{rijke17}.
Applied to the core inclusion $C^\simeq \to C$, this lets us deduce \cref{straightening} from descent.

Given a cocartesian fibration $p \co E \to X$, its directed univalent completion ought to be the full image factorisation of its straightening $X \to \cat_\infty$.
In our setting, we do not want to assume that $\cat_\infty$ exists, since we are trying to construct it, but we can still turn this idea into a \emph{construction} of the directed univalent completion, by playing out the directed join construction (\cf \cref{virtual-join-correctness}).
In the undirected setting, this idea was recently used by Uemura to define univalent completion of families~\cite{uemura25}.

\subsection*{Acknowledgements}

We thank Dani{\"e}l Apol for helpful discussions.

\section{Preliminaries}

In this section we describe some of the basic notions and results pertaining to spaces and categories used in the rest of the paper.
For a more thorough development of synthetic category theory we refer to work in progress of Cisinski, Cnossen, Nguyen, and Walde~\cite{cisinski25}.
We allow ourselves to work at a high level of abstraction and usually neglect, for example, to build witnesses that certain diagrams commute, with the hope that the reader can fill in the details in whatever precise setting%
\footnote{E.g., quasicategories, categories internal to an $(\infty,1)$-topos, or simplicial type theory.}
they are interested in.

We take as primitive the notion of space and the space $\cat^\simeq$ of categories.
For categories $C$ and $D$ we write $\map(C, D)$ for the space of functors $C \to D$.
We assume that functors can be composed and that functor composition is associative up to homotopy.
Thus categories form what is sometimes called a \emph{wild} category.%
\footnote{%
Here we use the word `wild' to emphasize that we never need infinitely many coherences; we have not investigated how much coherence is needed.
}
We assume that there is a terminal category $1$.
We assume that we may form pullbacks, pushouts, and functor categories.
We write the latter as $\fun(C, D)$.
We denote $\map(1, C)$ by $\ob(C)$ and refer to it as the \emph{space of objects} of $C$.
We take $c : C$ to mean $c : \ob(C)$.

We write $\II$ for the category that is freely generated by
objects $0, 1 : \II$ and a morphism $0 \to 1$.
We write $\arr(C)$ for the functor category $\fun(\II, C)$.
We write mapping spaces in $C$ as $C(x, y)$.
Morphisms in functor categories are called natural transformations and are depicted as 2-cells.

\subsection{Comma and slice categories}

Given functors $F \co B \to A$ and $G \co C \to A$ with common codomain, the comma category $F \downarrow G$ is given by the pullback of categories
\[
\begin{tikzcd}
  F \downarrow G
  \ar[r]
  \ar[d]
  \arrow[dr, phantom, "\lrcorner" very near start]
&
  \arr(A)
  \ar[d, "{(\dom, \cod)}"]
\\
  B \times C
  \ar[r, "F \times G"]
&
  A \times A \rlap{.}
\end{tikzcd}
\]
Equivalently, the comma category is cofreely generated by the following lax square.
\[
\begin{tikzcd}
  F \downarrow G
  \ar[r]
  \ar[d]
  \twocellflip{dr}
&
  C
  \ar[d, "G"]
\\
  B
  \ar[r, "F"']
&
  A \rlap{.}
\end{tikzcd}
\]
In the comma notation, we use categories as a stand-in for the identity functor on them.
For example, $A \downarrow G$ and $F \downarrow A$ denote commas with the identity functor $A \to A$.
We also use objects as a stand-in for the functor from the terminal category selecting them.
For example, $a \downarrow G$ and $F \downarrow a$ denote commas with the functor $a \co 1 \to A$.
Given $F \co C \to A$ and $a : A$, we may also write $a \downarrow C$ for $a \downarrow F$ when $F$ is clear from context.
We write the slice $A \downarrow a$ also as $A / a$.

\subsection{The fundamental theorem of category theory}

We assume the \emph{fundamental theorem of category theory}, saying that $\II$ detects equivalences of categories.
Equivalently, a functor $F \co C \to D$ is invertible if and only if it is
\begin{itemize}
\item
\emph{surjective}: for every $d : D$, there is $c : C$ with $F(c) \simeq d$,
\item
\emph{fully faithful}: for $x, y : C$, the action $C(x, y) \to D(F x, F y)$ of $F$ on morphisms is invertible.
\end{itemize}

\subsection{Full and wide subcategories}

We assume that we may form subcategories and that they enjoy the expected mapping-in universal properties.
That is, if $C$ is a category and $P$ is a predicate on the objects of $C$ (\ie, a monomorphism $P \hookrightarrow \ob(C)$), then there is a category $C_P$ with a functor $i \co C_P \to C$ such that for any category $X$, the map
\[
\begin{tikzcd}
  \map(X, C_p)
  \ar[r]
&
  \map(X, C)
\end{tikzcd}
\]
is a monic, and its image is spanned by those functors $F \co X \to C$ such that for every object $x : X$, $F(x)$ lies in $P$ (in other words, $\ob(F) \co \ob(X) \to \ob(C)$ factors through $P \hookrightarrow \ob(C)$).
We call $C_P$ the \emph{full subcategory of $C$ spanned by $P$}.

We similarly assume that we may form \emph{wide} subcategories.
That is, let $C$ be a category and let $P$ be a predicate on \emph{morphisms} of $C$ which is closed under composition.%
\footnote{This includes \emph{nullary} composition, \ie, we assume that $P$ contains all isomorphisms of $C$.}
Then there is a category $C_P$ with a functor $i \co C_P \to C$ such that for any category $X$, the map
\[
\begin{tikzcd}
  \map(X, C_p)
  \ar[r]
&
  \map(X, C)
\end{tikzcd}
\]
is monic, and its image is spanned by those functors $F \co X \to C$ such that for every morphism $f$ in $X$, the map $F(f)$ lies in $P$.

\subsection{Adjunctions}

Given categories $C$ and $D$ with functors $L \co C \to D$ and $R \co D \to C$, an adjunction $L \dashv R$ consists of an equivalence of categories $e \co C \downarrow R \simeq L \downarrow D$ over $C \times D$.
We say that $e$ exhibits $L$ as the left adjoint of $R$.
We have that left adjoints are unique when they exist, and that $R$ has a left adjoint if and only if the category $c \downarrow R$ has an initial object for every $c : C$.
An equivalent way of exhibiting $L$ as left adjoint to $R$ is by giving a natural transformation $\eta \co \id_C \to RL$ such that $\eta_c \co c \to R L c$ is initial in $c \downarrow R$ for every $c$.

We also make use of adjunctions in the wild setting.

\subsection{Spaces, groupoids, and extensivity}

Given a category $C$, we think of the mapping space $\map(1, C)$ as the space of \emph{objects} of $C$ and denote it $\ob(C)$.
The wild functor $\ob$ has a fully faithful left adjoint, which sends a space $X$ to the groupoid $D X = \bigsqcup_X 1$ given by the coproduct%
\footnote{In more traditional contexts, the term \emph{coproduct} is reserved for the case when $X$ is a set.}
of $X$-many copies of the terminal category (equivalently, the tensor of $1$ with $X$).
We will usually leave this inclusion from spaces to categories implicit.%
\footnote{We are not aware of any established notation for it.}
This induces an equivalence between spaces and groupoids, \ie, categories in which every morphism is invertible.

We ask that the coproduct defining $D X$ is is \emph{extensive}, \ie, satisfies descent.
Explicitly, this means that the (wild) functor
\[
  \cat / D X \to \cat^X
\]
which computes the fibres of a category over $D X$ is invertible (and so its inverse is the left adjoint $\bigsqcup_X$).
This means that a category $C$ over a given groupoid is determined by its family of fibres, and similarly for a functor of categories over a groupoid.

\subsection{Left and right fibrations}

We say that a functor $D \to C$ is a left (\resp right) fibration if it is orthogonal against the left endpoint inclusion $0 \co 1 \to \II$ (\resp right endpoint inclusion $1 \co 1 \to \II$), in the sense that the square of spaces
\[
\begin{tikzcd}
  \map(\II, D)
  \ar[r]
  \ar[d, "\dom"']
&
  \map(\II, C)
  \ar[d, "\dom"]
\\
  \ob(D)
  \ar[r]
&
  \ob(C)
\end{tikzcd}
\]
is cartesian.
If that is the case, one can show that the square of \emph{categories}
\[
\begin{tikzcd}
  \arr(D)
  \ar[r]
  \ar[d, "\dom"']
&
  \arr(C)
  \ar[d, "\dom"]
\\
  C
  \ar[r]
&
  D
\end{tikzcd}
\]
is also cartesian (by showing that the cartesian gap map is surjective and fully faithful).
We denote the (wild) category of left (\resp right) fibrations over $C$ by $\lfib(C)$ (\resp $\rfib(C)$).

\subsection{Cofinal functors}

We say a functor $F \co C \to D$ is \emph{left} (\resp right) \emph{cofinal} if it is orthogonal against left (\resp right) fibrations.
We assume that every functor $F \co C \to D$ factors as a left cofinal functor $C \to E$ followed by a left fibration $E \to D$.
It follows that this factorisation is unique.
Dually, every functor factors uniquely as a right cofinal functor followed by a right fibration.

\subsection{Left Kan extension} \label{subsection-lke}

For a left fibration $p \co E \to C$ and an arbitrary functor $F \co C \to D$, we may factor the composite $F p \co E \to D$ as a left cofinal map $E \to \overline{E}$ followed by a left fibration $\overline{E} \to D$.
This defines a left adjoint $F_! \co \lfib(C) \to \lfib(D)$ to the base-change functor $F^* \co \lfib(D) \to \lfib(C)$ (which sends a left fibration over $D$ to its pullback along $F$).

Left Kan extension admits a useful fibrewise description.
For $E \to C$ a left fibration and $d : D$ an arbitrary object, we have a functor $E \downarrow d \to \overline{E}_d$ to the fibre of $\overline{E} \to D$ over $d$.
The relevant assertion is that this functor exhibits $\overline{E}_d$ as the free groupoid on $E \downarrow d$, \ie, the localisation at all its arrows.
Dually, if $E \to C$ is a right fibration then we have a functor $d \downarrow E \to \overline{E}_d$ which exhibits $\overline{E}_d$ as the free groupoid on $d \downarrow E$.

Abstractly, we have the lax square of categories
\[
\begin{tikzcd}
  E \downarrow d
  \ar[r, "\cod"]
  \ar[d, "\dom"']
  \twocellflip{dr}
&
  1
  \ar[d, "d"]
\\
  C
  \ar[r, "F"']
&
  D \rlap{.}
\end{tikzcd}
\]
This induces a lax square
\[
\begin{tikzcd}
  \lfib(D)
  \ar[r, "F^*"]
  \ar[d, "d^*"']
  \twocellflip{dr}
&
  \lfib(C)
  \ar[d, "\dom^*"]
\\
  \lfib(1)
  \ar[r, "\cod^*"']
&
  \lfib(E \downarrow d) \rlap{,}
\end{tikzcd}
\]
and the relevant assertion is that its mate $\cod_! \dom^* \to d^* F_!$ is invertible.
More generally, the same property holds for the lax square coming from \emph{any} comma category: for functors $F \co A \to B$ and $G \co C \to B$, we have $G^* F_! \simeq \cod_! \dom^*$, where $\dom \co F \downarrow G \to A$ and $\cod \co F \downarrow G \to C$.

Dually, for restriction of left Kan extension of right fibrations, we have $G^* F_! \simeq \dom_! \cod^*$ where $\cod \co G \downarrow F \to A$ and $\dom \co G \downarrow F \to C$.

\subsection{The Yoneda lemma}

Let $C$ be a category with an initial object $c$.
We assume the Yoneda lemma in the following form: the functor $c \co 1 \to C$ is left cofinal.
In particular, for a category $C$ with an arbitrary object $c$, we can form the coslice category $c \downarrow C$ and observe that it has an initial object given by the identity on $c$.
So the corresponding functor $1 \to c \downarrow C$ is left cofinal.
Since $\cod \co c \downarrow C \to C$ is a left fibration, this means that it gives the unique (left cofinal, left fibration) factorisation of $c \co 1 \to C$.

In particular, suppose $F \co C \to D$ is some functor.
Since $1 \to c \downarrow C$ and $1 \to F c \downarrow D$ are both left cofinal, by 2-out-of-3 the induced functor $c \downarrow C \to F c \downarrow D$ is also left cofinal.
This exhibits $F c \downarrow D$ as the left Kan extension of $c \downarrow C$ along $F \co C \to D$.
Dually, the inclusion of a terminal object is right cofinal, and the left Kan extension of $C / c$ is $D / F c$.

Since the fibre of $D \downarrow F c$ over $d : D$ is the mapping space $D(d, F c)$, this lets us express mapping spaces of $D$ in terms of left Kan extension $F_! \co \rfib(C) \to \rfib(D)$ applied to representable right fibrations.

\subsection{Descent for left fibrations}

We assume that left fibrations descend along pushouts.
This means that for any pushout square of categories
\[
\begin{tikzcd}
  A
  \ar[r, "g"]
  \ar[d, "f"']
&
  B
  \ar[d, "l"]
\\
  C
  \ar[r, "k"']
&
  D \rlap{,}
  \ar[ul, phantom, "\ulcorner" very near start]
\end{tikzcd}
\]
the induced square
\[
\begin{tikzcd}
  \lfib(D)
  \ar[r, "k^*"]
  \ar[d, "l^*"']
&
  \lfib(C)
  \ar[d, "f^*"']
\\
  \lfib(B)
  \ar[r, "g^*"']
&
  \lfib(A)
\end{tikzcd}
\]
is a pullback square of categories.
This is a consequence of straightening--unstraightening for left fibrations.

\subsection{Limits and colimits in functor categories}

Given categories $C$ and $D$, we assume that the family of evaluation functors $\ev_c \co \fun(C, D) \to D$ for $c : C$ creates (co)limits.
This means that (co)limits are given levelwise when they exist at every level.
Such levelwise (co)limits are preserved under restriction $f^* \co \fun(C, D) \to \fun(C', D)$ along any functor $f \co C' \to C$.

In practice, we apply this only to describe pushouts, pullbacks, sequential colimits, and initial objects.
We apply it also in $\lfib(C)$, \ie, $\fun(C, \spc)$.
The relevant claim is then that for $f \co C' \to C$, restriction $f^* \co \lfib(C) \to \lfib(C')$ preserves pushouts and sequential colimits.

\subsection{Mapping spaces of sequential colimits}

A \emph{sequence} in a (wild) category $C$ consists of a family of objects $a \co \nat \to \ob(C)$ with morphisms $f_n \co a_n \to a_{n+1}$ for $n : \nat$.
A sequential colimit of such a sequence $a_0 \xrightarrow{f_0} x_1 \xrightarrow{f_1} \cdots$ is an initial cocone under it.

Sequential colimits are filtered colimits, so commute with pullbacks in spaces, and hence also in functor categories into spaces such as $\lfib(C)$ and $\rfib(C)$.

Consider a sequence of categories $A_0 \xrightarrow{F_0} A_1 \xrightarrow{F_1} \cdots$ with colimit $A_\infty$ with coprojections $i_n \co A_n \to A_\infty$.
Since sequential colimits are filtered, mapping spaces of $A_\infty$ admit a simple description: for objects $x, y : A_0$, we have a cocone under
\[
\begin{tikzcd}
  A_0(x, y)
  \ar[r]
&
  A_1(F_0 x, F_0 y)
  \ar[r]
&
  \cdots
\end{tikzcd}
\]
with point $A_\infty(i_0x, i_0y)$ given by the actions of $i_n$ on morphisms, and this exhibits the space $A_\infty(i_0 x,i_0 y)$ as a sequential colimit.
One can use this fact to show that $\ob$ and $\map(\II,-)$ commute with sequential colimits.

\subsection{Cocartesian fibrations}

Given a functor $p \co E \to B$ and a morphism $f \co x \to y$ in $E$, we say that $f$ is \emph{$p$-cocartesian} if the square of spaces
\[
\begin{tikzcd}
  E(y, z)
  \ar[r, "- \circ f"]
  \ar[d]
&
  E(x, z)
  \ar[d]
\\
  B(p y, p z)
  \ar[r, "- \circ p f"']
&
  B(p x, p z)
\end{tikzcd}
\]
is cartesian for all $z : E$.
Let $\arrcocart(p) \subseteq \arr(E)$ denote the full subcategory spanned by $p$-cocartesian maps.
One can show that the cartesian gap map of the square of categories
\[
\begin{tikzcd}
  \arrcocart(p)
  \ar[r, "\dom"]
  \ar[d]
&
  E
  \ar[d, "p"]
\\
  \arr(B)
  \ar[r, "\dom"]
&
  B
\end{tikzcd}
\]
is fully faithful.
We say $p$ is a \emph{cocartesian fibration} if that gap map is invertible.
Equivalently, the gap map is surjective, \ie, for objects $x : E$ and $y : B$ with $f : B(p x, y)$, there is a lift $y' : E$ of $y$ with a $p$-cocartesian lift $f' : E(x, y')$ of $f$.

Cocartesian fibrations are stable under base change.
That is, for $p : E \to B$ a cocartesian fibration and $x \co C \to B$ an arbitrary functor, the base change $x^* p \co E \times_B C \to C$ of $p$ is also a cocartesian fibration.
We might denote $E \times_B C$ as $x^* E$ in this context.
A map in $x^* E$ is $x^* p$-cocartesian if and only if its image in $E$ is $p$-cocartesian.

For cocartesian fibrations $p_0 \co E_0 \to B$ and $p_1 \co E_1 \to B$ with base $B$, we say that a functor $F \co E_0 \to E_1$ of categories over $B$ is \emph{cocartesian} if it sends $p_0$-cocartesian maps to $p_1$-cocartesian maps.
We denote by $\cocartfib(B)$ the resulting (wild) category of cocartesian fibrations over $B$.
If a cocartesian functor $F \co E_0 \to E_1$ over $B$ is \emph{fibrewise} fully faithful, then it is fully faithful.
This means that invertibility of cocartesian functors can be detected fibrewise.

\subsection{Cocartesian transport}
\label{cocartesian-transport}

Given a cocartesian fibration $p \co E \to B$ and a category $C$, we have a correspondence between functors $C \to \arrcocart(p)$ and functors $C \to \arr(B) \times_B E$.
The data of a functor $C \to \arrcocart(p)$ is equivalently the data of two functors $F, G \co C \to E$ with a natural transformation $\alpha \co F \to G$ all of whose components are $p$-cocartesian.
The data of a functor $C \to \arr(B) \times_B E$ is the data of functors $F \co C \to E$ and $G_0 \co C \to B$ and a natural transformation $\alpha_0 \co p F \to G_0$.
So given the latter data, we have a unique lift $G \co C \to E$ of $G_0$ together with a lift $\alpha \co F \to G$ of $\alpha_0$, such that the components of $\alpha$ are $p$-cocartesian.

In particular, pulling back a cocartesian fibration is covariantly functorial.
That is, given a category $C$ (\eg, the terminal category) and functors $x, y \co C \to B$ with a natural transformation $f \co x \to y$, we obtain a functor $f_! \co x^* E \to y^* E$ of cocartesian fibrations over $C$.
The image of $f_!$ in $E$ has cocartesian components and lifts $f$ in the appropriate sense.
Thus, the assignment $x \mapsto x^* E$ defines a (wild) functor $\fun(C, B) \to \cocartfib(C)$.

\subsection{Descent in spaces and left fibrations}

An excellent property of the category $\spc$ of spaces is that it has all colimits and these satisfy descent.
Abstractly, this means that given some diagram of spaces $X \co J \to \spc$, the functor
\[
\begin{tikzcd}
  \spc / \colim_J X
  \ar[r]
&
  \displaystyle\lim_{j : J} \spc / X_j
\end{tikzcd}
\]
given informally by base-change is an equivalence of categories.
Since this functor has a left adjoint (given informally by $\colim_J$), this says that the unit and counit of the adjunction are invertible.
Invertibility of the counit means that colimits are \emph{universal}, \ie, stable under base change.
Invertibility of the unit means that the wide subcategory of $\arr(\spc)$ spanned by \emph{cartesian squares} is closed under colimits.

Since colimits and pullback in $\lfib(C)$ are computed pointwise, colimits in $\lfib(C)$ also satisfy descent.

We use descent only for pushouts and sequential colimits.

\subsection{Mapping spaces of cocommas}

The \emph{cocomma} of a span of categories $B \xleftarrow{f} A \xrightarrow{g} C$ is the pushout of categories
\[
\begin{tikzcd}[column sep=3em]
  A \sqcup A
  \ar[r, "{(\id, \bracks{0, 1})}"]
  \ar[d, "f \sqcup g"']
&
  B \sqcup C
  \ar[d, "\bracks{k, l}"]
\\
  A \times \II
  \ar[r]
&
  \cocomma{f}{g} \rlap{.}
\end{tikzcd}
\]
More usefully, this means that the cocomma is freely generated by the lax square
\[
\begin{tikzcd}
  A
  \ar[r, "g"]
  \ar[d, "f"']
  \twocellflip{dr}
&
  C
  \ar[d, "l"]
\\
  B
  \ar[r, "k"']
&
  \cocomma{f}{g} \rlap{.}
\end{tikzcd}
\]

By descent for left fibrations, the lax square
\[
\begin{tikzcd}
  \lfib(\cocomma{f}{g})
  \ar[r, "k^*"]
  \ar[d, "l^*"']
  \twocellflip{dr}
&
  \lfib(B)
  \ar[d, "f^*"]
\\
  \lfib(C)
  \ar[r, "g^*"']
&
  \lfib(A)
\end{tikzcd}
\]
exhibits $\lfib(\cocomma{f}{g})$ as a comma category.
There is a general construction for finding left adjoints to the projections from a comma category, using in this case that $g^*$ has a left adjoint and that $\lfib(B)$ has an initial object preserved by $f^*$.
In this case, we get that the left adjoints $k_!$ and $l_!$ respectively classify the lax squares
\begin{align*}
\begin{tikzcd}[ampersand replacement=\&]
  \lfib(B)
  \ar[r, "g_! f^*"]
  \ar[d, equals]
  \twocellflip{dr}["\eta f^*" pos=0.3,outer sep=-1pt]
\&
  \lfib(C)
  \ar[d, "g^*"]
\\
  \lfib(B)
  \ar[r, "f^*"']
\&
  \lfib(A)
\end{tikzcd}
&&
\begin{tikzcd}[ampersand replacement=\&]
  \lfib(C)
  \ar[r, equals]
  \ar[d, "0"]
  \twocellflip{dr}["{!}" pos=0.3]
\&
  \lfib(C)
  \ar[d, "g^*"]
\\
  \lfib(B)
  \ar[r, "f^*"']
\&
  \lfib(A)
\end{tikzcd}
\end{align*}
In particular, the units $\id_{\lfib(B)} \to k_! k^*$ and $\id_{\lfib(C)} \to l_! l^*$ are invertible, meaning that the inclusions $k$ and $l$ into the cocomma are fully faithful.
Moreover, the mate $g_! f^* \to l^* k_!$ is invertible, and $k^* l_! \co \lfib(C) \to \lfib(B)$ is the initial functor.

\subsection{Localisation and pullback}

The inclusion from spaces to categories admits a left adjoint $\verts{-}$, which we refer to as \emph{localisation}.
This for example shows up when computing left Kan extensions.
Localization preserves finite products.
To see this the product of categories $B$ and $C$, one uses that $- \times C$ admits a right adjoint $\fun(C, -)$ which preserves groupoids.

\begin{corollary}\label{localisation-pullback}
Consider a cartesian square of categories
\[
\begin{tikzcd}
  D
  \ar[r]
  \ar[d]
&
  C
  \ar[d]
\\
  B
  \ar[r]
&
  A
\end{tikzcd}
\]
where $A$ is a groupoid.
Then the square remains cartesian after localisation.
\end{corollary}
\begin{proof}
By extensivity, we may work fibrewise over $a : A$.
The claim then reduces to the fact localization preserves the product of $B_a$ and $C_a$.
\end{proof}

\subsection{Left adjoints from arrow categories}

The following lemma gives a simple description of left adjoints from arrow categories.

\begin{lemma}\label{arr-left-adj}
Consider a functor $m \co C \to \arr(D)$, corresponding to a natural transformation $f^* \to g^*$ of functors $f^*, g^* \co C \to D$.
Suppose $f^*$ and $g^*$ have left adjoints $f_!, g_! \co D \to C$ and that $C$ has pushouts.
Then $m \co C \to \arr(D)$ has a left adjoint given by pushout
\[
\begin{tikzcd}
  g_! \dom
  \ar[r]
  \ar[d]
&
  g_! \cod
  \ar[d]
\\
  f_! \dom
  \ar[r]
&
  m \rlap{.}
  \arrow[ul, phantom, "\ulcorner" very near start]
\end{tikzcd}
\]
\end{lemma}
\begin{proof}
Given an object $x \to y$ of $\arr(D)$ and $c : C$, we have to show that the space of dashed square completions
\[
\begin{tikzcd}
  x
  \ar[r, dashed]
  \ar[d]
&
  f^* x
  \ar[d]
\\
  y
  \ar[r, dashed]
&
  g^* c
\end{tikzcd}
\]
is naturally equivalent to the space of dashed square completions
\[
\begin{tikzcd}
  g_1 x
  \ar[r]
  \ar[d]
&
  g_! y
  \ar[d, dashed]
\\
  f_! x
  \ar[r, dashed]
&
  c \rlap{.}
\end{tikzcd}
\]
This is straightforward.
\end{proof}

\subsection{2-limits and 2-colimits of categories}

Limits and colimits of categories enjoy $(\infty, 2)$-categorical universal properties.
Consider for example the coproduct of categories $C$ and $D$.
A priori, it has a $(\infty,1)$-categorical universal property, which expresses that the mapping space $\map(C \sqcup D, X)$ is equivalent to the product of mapping spaces $\map(C, X) \times \map(D, X)$.
But more is true: we have an equivalence of \emph{categories} $\fun(C \sqcup D, X) \simeq \fun(C, X) \times \fun(D, X)$.
To see this, note that $\map(Y, \fun(C \sqcup D, X))$ and $\map(Y, \fun(C, X) \times \fun(D, X))$ are both equivalent to $\map(C, \fun(Y, X)) \times \map(D, \fun(Y, X))$.
Similarly, we have that $\fun(X, C \times D) \simeq \fun(X, C) \times \fun(X, D)$.

We will make implicit use of this type of result for constructions like comma categories and localisations of categories.

\subsection{Cocartesian fibrations and the Conduch\'e condition}

We say that a functor $p \co E \to B$ is a \emph{Conduch\'e} fibration%
\footnote{%
It is true but not direct that a functor $p$ is a Conduch{\'e} fibration if and only if it is exponentiable, in the sense that base change along $p$ admits a right adjoint~\cite{ayala20}.
We do not use this.
}
if for objects $x, z : E$ and $b : B$, the square of spaces
\[
\begin{tikzcd}
  \verts{x \downarrow \fib_p(b) \downarrow z}
  \ar[r]
  \ar[d]
&
  E(x, z)
  \ar[d]
\\
  B(p x, b) \times B(b, p z)
  \ar[r]
&
  B(p x, p z)
\end{tikzcd}
\]
is cartesian.
Here $x \downarrow \fib_p(b) \downarrow z$ denotes the iterated comma category whose objects are given by a lift $y : E$ of $b$ with morphisms $x \to y$ and $y \to z$.

\begin{lemma}\label{conduche-of-cocartesian}
Every cocartesian fibration is a Conduch\'e fibration.
\end{lemma}

\begin{proof}
Consider a cocartesian fibration $p \co E \to B$.
The comma category $x \downarrow \fib_p(b)$ has a coreflective subcategory spanned by $p$-cocartesian maps $x \to y$.
This subcategory is equivalent to the space $B(p x,b)$.
In this way one shows that the localisation $\verts{x \downarrow \fib_p(b) \downarrow z}$ is equivalent to $B(p x, b) \downarrow z$, where the functor $B(p x, b) \to E$ sends a map $f \co p x \to b$ to $f_! x$.
The claim follows from taking fibres over $B(p x,b)$ and using that $x \to f_! x$ is cocartesian.
\end{proof}

\begin{lemma}\label{invertible-transport}
Consider a cocartesian fibration $p \co E \to B$ with a morphism $f : B(b, c)$ such that the cocartesian fibre transport functor $f_! \co E_b \to E_c$ is inverible.
Then $f$ has all $p$-cartesian lift, and a lift of $f$ is $p$-cartesian if and only if it is $p$-cocartesian.
\end{lemma}
\begin{proof}
Suppose $\overline{f} \co y \to z$ is a $p$-cocartesian lift of $f$; we claim that $\overline{f}$ is also $p$-cartesian.
Thus suppose $x : E$ lying over $a : B$.
We want to show that $E(x, y) \to E(x, z) \times_{B(a, c)} B(a, b)$ is invertible.
Let $g : B(a, b)$; we need to show that the map of spaces from the fibre of $E(x, y) \to B(a, b)$ over $g$ to the fibre of $E(x, z) \to B(a, c)$ over $f g$ is invertible.
The first space is equivalent to the mapping space $\fib_q(b)(g_! x, y)$ and the second is equivalent to $\fib_q(c)((f g)_! x, z) \simeq \fib_q(c)(f_! g_! x, f_! y)$.
The map is equivalent to the action of $f_!$, which is inverible by assumption.
Thus $\overline{f}$ is $p$-cartesian.

Since $f_!$ is surjective, every lift of $c$ is the target of some cocartesian lift of $f$.
This lift is then also cartesian, so $f$ has all cartesian lifts.
Uniqueness of cartesian lifts means that any cartesian lift is cocartesian.
\end{proof}

\section{Pulling back a reflective subcategory}

In this short section we describe a simple method, due to Kelly in the setting of ordinary categories~\cite{kelly80}, for pulling back a reflective subcategory $C \hookrightarrow A$ along a right adjoint $f^* \co B \to A$.
That is, we describe a formula for a left adjoint to the inclusion $C \times_A B \hookrightarrow B$.
Note that $C \times_A B$ is the full subcategory of $B$ spanned by objects $b$ such that $f^* b$ lies in the image of $C \hookrightarrow A$.

We are particularly interested in pulling back the reflective subcategory $\rfib(W) \hookrightarrow \arr(\rfib(W))$, spanned by isomorphisms, along a functor $m^* \co \rfib(C) \to \arr(\rfib(W))$; this will allow us to compute mapping spaces in localisations of $C$.

\begin{construction}[label=S-constr]
Consider categories $A, B, C$ with functors $f^* \co B \to A$ and $g^* \co C \to A$.
Suppose that $f^*$ and $g^*$ have respective left adjoints $f_!$ and $g_!$, and that $B$ has pushouts.
We obtain an endofunctor $S \co B \to B$ with a pointing $s \co \id_B \to S$ by the following pushout square:
\[
\begin{tikzcd}
  f_! f^*
  \ar[r, "\varepsilon_f"]
  \ar[d, "f_! \eta_g f^*"']
&
  \id_B
  \ar[d]
\\
  f_! g^* g_! f^*
  \ar[r]
&
  \ar[ul, phantom, "\ulcorner" very near start]
  S \rlap{.}
\end{tikzcd}
\]
\end{construction}

Note that the above square transposes under the adjunction $f_! \dashv f^*$ to the square
\begin{equation}\label{tr-square}
\begin{tikzcd}
  f^*
  \ar[r, "\id"]
  \ar[d, "\eta_g f^*"']
&
  f^*
  \ar[d, "f^* s"]
\\
  g^* g_! f^*
  \ar[r]
&
  f^* S \rlap{.}
\end{tikzcd}
\end{equation}

\begin{lemma}[label=s-ortho]
Consider a cospan $B \xrightarrow{f^*} A \xleftarrow{g^*} C$ as in \cref{S-constr}.
Given $x, y : B$ with $f^* y$ lying in $C$, the map $s_x \co x \to S x$ is orthogonal against $y$, in the sense that the map $B(Sx, y) \to B(x, y)$ of spaces is invertible.
\end{lemma}
\begin{proof}
Since $s_x$ is a cobase change of $f_! \eta_g f^* x$, it suffices to show that $f_! \eta_g f^* x$ is orthogonal against $y$.
Transposing under the adjunction $f_! \dashv f^*$, it suffices to show that $\eta_g f^* x$ is orthogonal against $f^* y$.
This is immediate since $f^* y$ lies in $C$.
\end{proof}

\begin{construction}[label=Sinfty-constr]
Consider a cospan of categories $B \xrightarrow{f^*} A \xleftarrow{g^*} C$ as in \cref{S-constr}.
Suppose that $B$ has sequential colimits.
Define $S^\infty \co B \to B$ as the sequential colimit of the sequence
\[
\begin{tikzcd}
  \id_B
  \ar[r, "s"]
&
  S
  \ar[r, "s S"]
&
  S S
  \ar[r, "s S S"]
&
  \cdots \rlap{.}
\end{tikzcd}
\]
We consider $S^\infty$ to be pointed by the sequential composition $\id_B \to S^\infty$.
\end{construction}

\begin{lemma}[label=Sinfty-in-sub]
Consider a cospan of categories $B \xrightarrow{f^*} A \xleftarrow{g^*} C$ as in \cref{S-constr}.
Suppose that $B$ has sequential colimits, that $f^*$ preserves them, and that $C \hookrightarrow A$ is closed under sequential colimits.
Then $S^\infty$ takes values in $C \times_A B \hookrightarrow B$.
\end{lemma}
\begin{proof}
Let $b : B$.
We have to show that $f^*S^\infty b$ lies in $C$.
Since $f^*$ preserves sequential colimits, this is the colimit of the sequence
\[
\begin{tikzcd}
  f^* b
  \ar[r, "f^* s_b"]
&
  f^* S b
  \ar[r, "f^* s_{S b}"]
&
  f^* S S b
  \ar[r, "f^* s_{S S b}"]
&
  \cdots \rlap{.}
\end{tikzcd}
\]
By the square~\eqref{tr-square}, each map in this sequence factors through an object of $C$.
The colimit is thus equivalently a sequential colimit of objects of $C$, so it lies in $C$ since $C$ is closed under sequential colimits.
\end{proof}

\begin{corollary}[label=Sinfty-reflector]
Consider a cospan of categories $B \xrightarrow{f^*} A \xleftarrow{g^*} C$ with adjoints $f_! \dashv f^*$ and $g_! \dashv g^*$.
Assume that $g^*$ is fully faithful, $B$ has pushouts and sequential colimits, $f^*$ preserves sequential colimits, and $C \hookrightarrow A$ is closed under sequential colimits.
Then the natural transformation $\id_B \to S^\infty$ from \cref{Sinfty-constr} exhibits $S^\infty$ as a reflector onto the subcategory $C \times_A B \hookrightarrow B$.
\end{corollary}
\begin{proof}
Combine \cref{s-ortho,Sinfty-in-sub}.
\end{proof}

\begin{example}[label=ex-localisation]
Let $C$ be a category and $m \co W \to \arr(C)$ an arbitrary functor.
Consider the pushout square of categories
\[
\begin{tikzcd}
  W \times \II
  \ar[r, "\fst"]
  \ar[d, "\overline{m}"']
&
  W
  \ar[d, "j"]
\\
  C
  \ar[r, "i"']
&
  {C[W^{-1}]} \rlap{.}
  \ar[ul, phantom, "\ulcorner" very near start]
\end{tikzcd}
\]
Then $i \co C \to C[W^{-1}]$ is the localisation of $C$ at the collection of morphisms $m w$ for $w : W$.

By descent for right fibrations, the induced square of categories
\[
\begin{tikzcd}
  \rfib(C[W^{-1}])
  \ar[r, hook, "i^*"]
  \ar[d, "j^*"']
&
  \rfib(C)
  \ar[d, "m^*"]
\\
  \rfib(W)
  \ar[r, hook, "\Delta"']
&
  \arr(\rfib(W))
\end{tikzcd}
\]
is cartesian.
The reflector of $\Delta$ is $\cod \co \arr(\rfib(W)) \to \rfib(W)$.
All the conditions of \cref{Sinfty-reflector} are satisfied, so we get a description of the reflector $i_! \co \rfib(C) \to \rfib(C[W^{-1}])$ as the sequential colimit of a sequence $\id \xrightarrow{s} S \xrightarrow{s S} S S \to \cdots$.

Writing $m$ as a natural transformation $k \to l$ of functors $k, l \co W \to C$, restriction $m^*$ corresponds to the map $l^* \to k^*$, and so we can use \cref{arr-left-adj} to obtain a more explicit description of $m_!$ and hence of the endofunctor $S$.
Namely, we have $m_! \Delta \simeq l_!$, and we have the following pushout squares
of endofunctors on $\rfib(C)$:
\begin{align*}
\begin{tikzcd}[ampersand replacement=\&]
  k_! l^*
  \ar[r]
  \ar[d]
\&
  k_! k^*
  \ar[d]
\\
  l_! l^*
  \ar[r]
\&
  m_! m^* \rlap{,}
  \ar[ul, phantom, "\ulcorner" very near start]
\end{tikzcd}
&&
\begin{tikzcd}[ampersand replacement=\&]
  m_! m^*
  \ar[r]
  \ar[d]
\&
  \id_{\rfib(C)}
  \ar[d]
\\
  l_! k^*
  \ar[r]
\&
  S \rlap{.}
  \ar[ul, phantom, "\ulcorner" very near start]
\end{tikzcd}
\end{align*}
Given a presheaf $P : \rfib(C)$, we think of $SP$ as the result of freely non-recursively adding \emph{inverses} to the transport maps $P(l w) \to P(k w)$ for $w : W$.
Indeed an object of $(l_! k^* P)(c)$ is, roughly speaking, given by an object $w : W$ together with a map $f \co c \to lw$ and an element $p : P(k w)$; we think of this as the ``inverse'' to $P(l w) \to P(k w)$ at $p$ followed by restriction along $f$.
The objects $k_! l^*$, $k_! k^*$, and $l_! l^*$ in the pushout square defining $m_! m^*$ describe ways in which one can build an object of $(l_! k^* P)(c)$ that ought to live already in $P(c)$, by some inverse law.

An object of $(S^n P)(c)$ thus corresponds, roughly speaking, to an object of $P$ at some point $x : C$, together with a zigzag connecting $x$ with $c$, containing at most $n$ `inverse' zags, all of the form $k w \to l w$.
\end{example}

In particular the example above gives a description of mapping spaces (\ie, representable presheaves) of the localisation $C[W^{-1}]$.
We are particularly interested in the \emph{functoriality} of this description.
To this end, we employ the organisational device of \emph{gluing}.

\section{Gluing a functor}

Given a functor $p \co A \to B$, the \emph{gluing} $\gl(p)$ is the comma category $B \downarrow p$.%
\footnote{This construction is also known as \emph{Artin gluing}.}
In other words, $\gl(p)$ is cofreely generated by functors $\pi_1 \co \gl(p) \to B$ and $\pi_2 \co \gl(p) \to A$ with a natural transformation $\alpha \co \pi_1 \to p \pi_2$.
In this section, we recall some basic results on this construction.

\begin{lemma}[label=glue-colimit]
Let $p \co A \to B$ and suppose $A$ and $B$ have colimits of shape $J$.
Then so does $\gl(p)$, and $\pi_1$, $\pi_2$ preserve colimits of shape $J$.
\end{lemma}
\begin{proof}
Direct.
\end{proof}

\begin{construction}
Gluing is functorial in lax squares: given a lax square of categories
\[
\begin{tikzcd}
  A
  \ar[r]
  \ar[d, "p"]'
  \twocellflip{dr}
&
  C
  \ar[d, "q"]
\\
  B
  \ar[r]
&
  D \rlap{,}
\end{tikzcd}
\]
we obtain a functor $\gl(p) \to \gl(q)$ which classifies the pasting of the lax square with the lax triangle $\alpha \co \pi_1 \to p \pi_2$.
\end{construction}

\begin{lemma}[label=glue-adj]
Consider a commutative square of categories
\[
\begin{tikzcd}
  A
  \ar[r, "u"]
  \ar[d, "q"]
&
  C
  \ar[d, "p"']
\\
  B
  \ar[r, "v"']
&
  D
\end{tikzcd}
\]
with left adjoints $f \dashv u$ and $g \dashv v$.
The induced functor $U \co \gl(p) \to \gl(q)$ has a left adjoint $F \co \gl(q) \to \gl(p)$, which is induced by the mate of the square above.
\end{lemma}

In particular, the commutative squares
\begin{align*}
\begin{tikzcd}[ampersand replacement=\&]
  \gl(p)
  \ar[r, "\pi_1"]
  \ar[d, "U"']
\&
  B
  \ar[d, "v"]
\\
  \gl(q)
  \ar[r, "\pi_1"']
\&
  D \rlap{,}
\end{tikzcd}
&&
\begin{tikzcd}[ampersand replacement=\&]
  \gl(p)
  \ar[r, "\pi_2"]
  \ar[d, "U"']
\&
  A
  \ar[d, "u"]
\\
  \gl(q)
  \ar[r, "\pi_2"']
\&
  C \rlap{.}
\end{tikzcd}
\end{align*}
have invertible mates, \ie, the mates $g \pi_1 \to \pi_1 F$ and $f \pi_2 \to \pi_2 F$ are invertible.

\begin{proof}
Consider an object $x \co b \to p a$ of $\gl(p)$ and an object $y \co d \to q c$ of $\gl(q)$.
Maps $y \to U x$ are given by the dashed data in the square
\[
\begin{tikzcd}
  d
  \ar[rr]
  \ar[d, dashed]
&&
  q c
  \ar[d, dashed, "q(?)"]
\\
  v b
  \ar[r]
&
  v p a
  \ar[r, "\sim"]
&
  q u a \rlap{.}
\end{tikzcd}
\]
Maps $F y \to x$ are given by the dashed data in the square
\[
\begin{tikzcd}
  g d
  \ar[r]
  \ar[d]
&
  g q c
  \ar[r]
&
  p g c
  \ar[d, "p(?)"]
\\
  b
  \ar[rr]
&&
  p a \rlap{.}
\end{tikzcd}
\]
These are naturally equivalent.
\end{proof}

\subsection{Cartesian maps}

Given a functor $p \co A \to B$, we may understand the projection $\pi_2 \co \gl(p) \to A$ via the restriction of cocartesian fibrations
\[
\begin{tikzcd}
  \gl(p)
  \ar[r]
  \ar[d, "\pi_2"']
  \ar[dr, phantom, "\lrcorner" very near start]
&
  \arr(B)
  \ar[d, "\cod"]
\\
  A
  \ar[r, "p"']
&
  B \rlap{.}
\end{tikzcd}
\]
A morphism $x \to Y$ of $\arr(B)$ is said to be \emph{cartesian} if it is $\pi_2$-cartesian.
Equivalently, its image in $\arr(B)$ is cartesian, meaning that the following square in $B$ is cartesian:
\[
\begin{tikzcd}
  \pi_1 x
  \ar[r]
  \ar[d, "\alpha_x"']
&
  \pi_1 y
  \ar[d, "\alpha_y"]
\\
  q \pi_2 x
  \ar[r]
&
  q \pi_2 y \rlap{.}
\end{tikzcd}
\]

\begin{lemma}[label=big-list-lemma]\label{big-list}
Given a functor $p \co A \to B$ be a functor, the collection of cartesian maps of $\gl(p)$ the following closure properties.
\begin{parts}
\item\label{big-list:cart-comp}
Any isomorphism is cartesian.
Given maps $f \co x \to y$ and $g \co y \to z$ in $\gl(p)$ with $g$ cartesian, $f$ is cartesian if and only if $gf$ is cartesian.
\item\label{big-list:cart-func}
Given a commutative square $v p \simeq q u$ where $v$ preserves pullbacks, the induced functor $\gl(p) \to \gl(q)$ preserves cartesian maps.
\item\label{big-list:cart-cobase-change}
Suppose $A$ and $B$ have pushouts, that $p$ preserves them, and that pushouts in $B$ satisfy descent.
Then for any span $y \xleftarrow{f} x \xrightarrow{g}z$ in $\gl(p)$ with $f$ and $g$ cartesian, the maps $y \to y \sqcup^x z$, $z \to y \sqcup^x z$ are also cartesian.
\item\label{big-list:cart-cogap}
Suppose $A$ and $B$ have pushouts, that $p$ preserves them, and that pushouts in $B$ are stable under base change.
Then for any square in $\gl(p)$ with all four maps cartesian, the cogap map is also cartesian.
\item\label{big-list:cart-pushout}
Suppose $A$ and $B$ have pushouts, that $p$ preserves them, and that pushouts in $B$ satisfy descent.
Then for any map of spans in $\gl(p)$ with all seven maps involved cartesian, the induced map on pushouts is also cartesian.
\item\label{big-list:cart-transfinite}
Suppose $A$ and $B$ have sequential colimits, that $p$ preserves them, and that sequential colimits in $B$ satisfy descent.
Then cartesian maps in $\gl(p)$ are closed under transfinite composition.
\item\label{big-list:cart-sequential}
Suppose $A$ and $B$ have sequential colimits, that $p$ preserves them, and that sequential colimits commute with pullback in $B$.
Given a sequence of maps in $\gl(p)$
\[
\begin{tikzcd}
  x_0
  \ar[r]
  \ar[d, "f_0"']
&
  x_1
  \ar[r]
  \ar[d, "f_1"']
&
  \cdots
&
  x_\infty
  \ar[d, "f_\infty"]
\\
  y_0
  \ar[r]
&
  y_1
  \ar[r]
&
  \cdots
&
  y_\infty \rlap{,}
\end{tikzcd}
\]
if the vertical maps $f_n$ are all cartesian, so is the induced map $f_\infty$ on sequential colimits.
\end{parts}
\end{lemma}
\begin{proof}\leavevmode
\begin{parts}
\item
This is pasting of cartesian morphisms for the functor $\pi_2$.
\item
Say $q \co C \to D$.
In the square of categories
\[
\begin{tikzcd}
  \gl(p)
  \ar[r]
  \ar[d]
&
  \arr(B)
  \ar[d, "\arr(v)"]
\\
  \gl(q)
  \ar[r]
&
  \arr(D) \rlap{,}
\end{tikzcd}
\]
the horizontal functors create cartesian morphisms.
By assumption, the right functor preserves cartesian morphisms, hence so does the left functor.
\item
Since $p$ preserves pushouts, we may as well work in $\arr(B)$ rather than $\gl(p)$.
Here, the statement follows directly from descent.
\item
Same as above.
\item
Consider a diagram in $\gl(p)$
\[
\begin{tikzcd}
  b
  \ar[d]
&
  a
  \ar[l]
  \ar[r]
  \ar[d]
&
  c
  \ar[d]
\\
  y
&
  x
  \ar[l]
  \ar[r]
&
  z
\end{tikzcd}
\]
with all maps cartesian.
By \cref{big-list:cart-cobase-change}, the map $y \to y \sqcup^x z$ is cartesian, and so by \cref{big-list:cart-comp}, so is $b \to y \sqcup^x z$, and by symmetry so is $c \to y \sqcup^x z$.
We are now done by \cref{big-list:cart-cogap}.
\item
As above, this reduces to the claim in $\arr(B)$, where it is simply  the assumption that sequential colimits in $B$ satisfy descent.
\item
This reduces to the claim in $\arr(B)$, where it is simply the assumption that sequential colimits commute with pullback.
\end{parts}
\end{proof}

\subsection{Pulling back and gluing reflective subcategories}

Consider a map of cospans of categories,
\begin{equation}\label{map-of-inputs}
\begin{tikzcd}
  B_u
  \ar[r, "f_u^*"]
&
  A_u
&
  C_u
  \ar[l, hook', "g_u^*"']
\\
  B_d
  \ar[r, "f_d^*"']
  \ar[u, "q"']
&
  A_d
  \ar[u, "p"]
&
  C_d
  \ar[l, hook', "g_d^*" ]
  \ar[u, "r"']
\end{tikzcd}
\end{equation}
where each cospan (row) satisfies the assumptions of \cref{Sinfty-reflector}.
By \cref{glue-colimit,glue-adj}, we obtain a cospan
\[
\begin{tikzcd}
  \gl(q)
  \ar[r, "f^*"]
&
  \gl(p)
&
  \gl(r)
  \ar[l, hook', "g^*"']
\end{tikzcd}
\]
which again satisfies the assumptions of \cref{Sinfty-reflector}.

Denote the pullbacks $B_u \times_{A_u} C_u$ and $B_d \times_{A_d} C_d$ by $D_u$ and $D_d$, respectively, and the induced functor $D_d \to D_u$ by $t$.
Note that $\gl(t)$ --- the full subcategory of $\gl(q)$ spanned by objects whose components lie in $D_u$ and $D_d$ --- is equivalently the pullback $\gl(q) \times_{\gl(p)} \gl(r)$.
Moreover, the reflector $S^\infty \co \gl(q) \to \gl(t)$ is computed componentwise, by \cref{glue-adj}.

We would like to know when it happens that, for an object $x : \gl(q)$, the unit map $x \to S^\infty x$ is cartesian.
Since this decomposes as a transfinite composition
\[
\begin{tikzcd}
  x
  \ar[r]
&
  S x
  \ar[r]
&
  S S x
  \ar[r]
&
  \cdots \rlap{,}
\end{tikzcd}
\]
it is natural to look for conditions that ensure that each map in this sequence is cartesian.
By definition, $s_x \co x \to S x$ is the cobase change of $f_! \eta_g f^* x \co f_! f^* x \to f_! g^* g_! f^* x$ along $\varepsilon_f x \co f_! f^* x \to x$.
Say that $x : \gl(q)$ is \emph{good} if $\eta_g f^* x$ and $\varepsilon_f x$ are both cartesian.
We are thus led to ask if it happens that $S x$ is good as soon as $x$ is good.

\begin{lemma}[label=S-good]
Suppose we are given a diagram of categories as in~\eqref{map-of-inputs},
where both cospans satisfy the assumptions of \cref{Sinfty-reflector}.
Suppose moreover that:
\begin{conditions}
\item \label{S-good:descent}
Sequential colimits and pushouts in $B_u$ and $A_u$ satisfy descent.
\item \label{S-good:preserve-colimits}
The functors $p$, $q$, $f^*$, $g^*$ preserve sequential colimits and pushouts.%
\footnote{Note that if $f_u^*$, $f_d^*$, $g_u^*$, $g_d^*$ all preserve colimits, then so do $f^*$ and $g^*$.}
\item \label{S-good:g-cart}
$g_! \co \gl(p) \to \gl(r)$ preserves cartesian maps.
\item\label{S-good:f-cart}
$f_! \co \gl(p) \to \gl(q)$ preserves every cartesian map $x \to y$ such that $\eta_g y \co y \to g^* g_! y$ is cartesian.
\item\label{S-good:c-good}
For $c : C$, the object $f_! g^* c$ of $\gl(q)$ is good.
\end{conditions}
Then for any $x : \gl(q)$ which is good, we have that $Sx$ is also good, and the unit map $x \to S^\infty x$ is cartesian.
\end{lemma}

Explicitly, given objects $x_u : B_u$ and $x_d : B_d$ with a map $m \co x_u \to p x_d$, determining a good object $x : \gl(q)$, the statement that the unit map $x \to S^\infty x$ is cartesian means that
the square in $B_u$
\[
\begin{tikzcd}
  x_u
  \ar[r, "\eta"]
  \ar[d, "m"']
  \ar[dr, phantom, "\lrcorner" very near start]
&
  h_{u!} x_u
  \ar[d]
\\
  q x_d
  \ar[r, "q \eta"']
&
  q h_{d!} x_d
\end{tikzcd}
\]
is cartesian where $h_{u!}$ and $h_{d!}$ denote the reflectors $B_u \to D_u$ and $B_d \to D_d$, respectively.

\begin{proof}
Suppose $x : \gl(q)$ is good; we want to show that $Sx$ is also good.
That is, we want to show that $\eta_g f^* S x$ and $\varepsilon_f S x$ are cartesian.
Since both the domains and codomains of $\eta_g f^*$ and $\varepsilon_f$ preserve pushouts, in particular the pushout defining $Sx$, these maps are induced by the following maps of spans and we are in a position to apply \cref{big-list:cart-pushout} of \cref{big-list}.
We have that all labelled maps in the diagrams
\[
\begin{tikzcd}[column sep=5em]
  f^* f_! g^* g_! f^*x
  \ar[d, "\eta_g f^* f_! g^* g_! f^* x"']
&
  f^* f_! f^*x
  \ar[d]
  \ar[l]
  \ar[r, "f^* \varepsilon_f x"]
&
  f^* x
  \ar[d, "\eta_g f^* x"]
\\
  g^* g_! f^* f_! g^* g_! f^* x
&
  g^* g_! f^* f_! f^* x
  \ar[l, "g^* g_! f^* f_! \eta_g f^* x"]
  \ar[r, "g^* g_! f^* \varepsilon_f x"']
&
  g^* g_! f^* x
\end{tikzcd}
\]
and
\[
\begin{tikzcd}[column sep=5em]
  f_! f^* f_! g^* g_! f^*x
  \ar[d, "\varepsilon_f f_! g^* g_! f^* x"']
&
  f_! f^* f_! f^*x
  \ar[d]
  \ar[l, "f_! f^* f_! \eta_g f^* x"']
  \ar[r, "f^* \varepsilon_f x"]
&
  f_! f^* x
  \ar[d, "\varepsilon_f x"]
\\
  f_! g^* g_! f^* x
&
  f_! f^* x
  \ar[l, "f_! \eta_g f^* x"]
  \ar[r, "\varepsilon_f x"']
&
  g^* g_! f^* x
\end{tikzcd}
\]
are cartesian; it then follows by \cref{big-list:cart-comp} of \cref{big-list} that the unlabelled maps are cartesian, so that $Sx$ is good.

By assumption, $x$ is good, \ie, $\varepsilon_f x$ and $\eta_g f^* x$ are cartesian.
We have that $f^*$ and $g^*$ preserve cartesian maps by \cref{big-list:cart-func} of \cref{big-list}, and $g_!$ preserves them by assumption, so $f^* \varepsilon_f x$ and $g^* g_! f^* \varepsilon_f$ are cartesian.
To verify that $f_!$ preserves the cartesian map $\eta_g f^* x$, we apply \cref{S-good:f-cart}, noting that $\eta_g g^* g_! f^* x$ is invertible and so in particular cartesian.
We also want to know that $f_!$ preserves the cartesian map $f^* \varepsilon_f x$; indeed $\eta_g f^*x$ is cartesian by assumption.
The remaining two maps, $\eta_g f^*f_! g^* g_! f^*x$ and $\varepsilon_f f_! g^* g_! f^* x$, are both cartesian by \cref{S-good:c-good} with $c = g_! f^* x$.
This finishes the proof that if $x$ is good then so is $S x$.

If $x$ is good, then $s_x \co x \to S x$ is cartesian by \cref{big-list:cart-cobase-change} of \cref{big-list}.
Since $S x$ is also good, every map in the sequence $x \xrightarrow{s_x} S x \xrightarrow{s_{S x}} S S x \to \cdots$ defining $S^\infty x$ is good.
By \cref{big-list:cart-transfinite} of \cref{big-list}, the transfinite composition $x \to S^\infty x$ is also good, as needed.
\end{proof}

\section{Localising a cocartesian fibration}

Let $C$ be a category and $W$ a collection of morphisms in $C$.
In this section we prove descent for cocartesian fibrations along the localisation functor $i \co C \to C[W^{-1}]$.
Given a cocartesian fibration $q \co E \to C$, say that $q$ \emph{inverts} $W$ if for every morphism $f \co x \to y$ in $W$, the transport functor $f_! \co E_x \to E_y$ on fibres of $q$ is invertible.
Any cocartesian fibration over $C[W^{-1}]$ pulls back along $i$ to give a cocartesian fibration over $C$ that inverts $W$.
We would like to know that \emph{every} cocartesian fibration which inverts $W$ arises in this way.

Given a cocartesian fibration $q \co E \to C$ which inverts $W$, let $W_u$ be the collection of $q$-cocartesian lifts of morphisms in $W$.
The composite $E \to C \to C[W^{-1}]$ extends along $i_u \co E \to E[W_u^{-1}]$ to a functor $p \co E[W_u^{-1}] \to C[W^{-1}]$.
Most of this section is devoted to proving that $q'$ is a cocartesian fibration, and that $q$ is the pullback of $q'$. Both of these statements concern mapping spaces of the localisations $E[W_u^{-1}]$ and $C[W^{-1}]$.
We will see that they can be reduced to the following more concrete statements.
\begin{lemma}[label=mapping-spaces-cart]
Let $q \co E \to C$ be a cocartesian fibration which inverts $W$ and let $W_u$ be the collection of $q$-cocartesian lifts of morphisms in $W$.
\begin{parts}
\item\label{mapping-spaces-cart:arrow}
For objects $x, y : E$, the following square of spaces is cartesian:
\[
\begin{tikzcd}
  E(x, y)
  \ar[r]
  \ar[d]
  \ar[dr, phantom, "\lrcorner" very near start]
&
  E[W_u^{-1}](i_u x, i_u y)
  \ar[d]
\\
  C(q x, q y)
  \ar[r]
&
  C[W^{-1}](i q x, i q y) \rlap{.}
\end{tikzcd}
\]
\item\label{mapping-spaces-cart:zig-zag-start}
For objects $x, y : E$ and $c : C$, the following square is cartesian:
\[
\begin{tikzcd}
  \verts{x \downarrow E_c \downarrow i_u y}
  \ar[r]
  \ar[d]
  \ar[dr, phantom, "\lrcorner" very near start]
&
  E[W_u^{-1}](i_u x, i_u y)
  \ar[d]
\\
  C(q x, c) \times C[W^{-1}](i c, i q y)
  \ar[r]
&
  C[W^{-1}](i q x,i q y) \rlap{.}
\end{tikzcd}
\]
Here $x \downarrow E_c \downarrow i_u y$ denotes the iterated comma category whose objects consist of an object $z$ in the fibre of $p$ over $c$ together with maps $x \to z$ and $i_u z \to i_u y$.
\end{parts}
\end{lemma}

We will prove this by appealing to \cref{S-good}.
In fact, we establish a slightly more general result.
Say a Conduch\'e fibration $q \co E \to C$ inverts $W$ if morphisms in $W$ have all $q$-cocartesian and $q$-cartesian lifts and these coincide.%
\footnote{This can be read as saying that for a morphism $x \to y$ in $W$, the profunctor
$E_x \profto E_y$ is representable and invertible.}

\begin{lemma}\label{mapping-spaces-conduche}
Let $p \co E \to C$ be a Conduch\'e fibration which inverts $W$ and let $W_u$ be the collection of $q$-cocartesian (equivalently, $q$-cartesian) lifts of morphisms in $W$.
Then $p$ satisfies the conclusion of \cref{mapping-spaces-cart}.
\end{lemma}

We defer the proof of \cref{mapping-spaces-conduche} to the following subsection, but let us first remark that it is a direct generalisation of \cref{mapping-spaces-cart}.

\begin{proof}[Proof of \cref{mapping-spaces-cart} from \cref{mapping-spaces-conduche}]
By \cref{conduche-of-cocartesian,invertible-transport}, any cocartesian fibration that inverts $W$ is also a Conduch\'e fibration that inverts $W$.
\end{proof}

\subsection{Setting up the proof of the main lemma}

Throughout this subsection, we work in the context of \cref{mapping-spaces-conduche} and work toward its proof.

We start by establishing more systematic notation, which will be fixed until the lemma is proved.
Let us write $A_d$ and $A_u$ for $C$ and $E$; we systematically use subscripts $d$ and $u$ to denote things taking place at the level of $A_d$ and $A_u$, respectively.
So we write $W_d$ instead of $W$.
We also identify $W_d$ with its \emph{space} of morphisms.
Thus $W_d$ is a space with a functor $W_d \to \arr(A_d)$.
We denote the localisation functor $A_d \to A_d[W_d^{-1}]$ by $i_d$.
We denote by $W_u$ the full subcategory of $W_d \times_{\arr(A_d)} \arr(A_u)$ spanned by objects whose underlying arrow in $A_u$ is $q$-cocartesian (equivalently, $q$-cartesian).
As before, we denote the functor $A_u[W_u^{-1}] \to A_d[W_d^{-1}]$ by $q'$.
We denote the functor $W_u \to W_d$ by $p$, and the functors $W_d \to \arr(A_d)$ and $W_u \to \arr(A_u)$ by $m_d$ and $m_u$, respectively.
We denote the composites of $m_d$ with $\dom$ and $\cod$ by $s_d, t_d \co W_d \to A_d$, respectively, and similarly define $s_u, t_u \co W_u \to A_u$.

We emphasize the asymmetry between $W_d$ and $W_u$: it is important that we consider $W_d$ as a groupoid and $W_u$ as a category.
Since $W_u$ consists precisely of $q$-cocartesian (equivalently, $q$-cartesian) lifts of $W_d$, both of the following squares of categories are cartesian:
\begin{align*}
\begin{tikzcd}[ampersand replacement=\&]
  W_u
  \ar[r, "s_u"]
  \ar[d, "p"']
  \ar[dr, phantom, "\lrcorner" very near start]
\&
  A_u
  \ar[d, "q"]
\\
  W_d
  \ar[r, "s_d"]
\&
  A_d \rlap{,}
\end{tikzcd}
&&
\begin{tikzcd}[ampersand replacement=\&]
  W_u
  \ar[r, "t_u"]
  \ar[d, "p"']
  \ar[dr, phantom, "\lrcorner" very near start]
\&
  A_u
  \ar[d, "q"]
\\
  W_d
  \ar[r, "t_d"]
\&
  A_d \rlap{.}
\end{tikzcd}
\end{align*}

In the lemma below, $q^*$ is understood to refer to restriction $\rfib(A_d) \to \rfib(A_u)$ of right fibrations along $q \co A_u \to A_d$.
Given an object $a_u : A_u$, we denote by $\yogl a$ the object of $\gl(q^*)$ given by the unit $\yo a_u \to q^* q_! \yo a_u$ where $\yo a_u$ denotes the representable right fibration $\dom \co A_u / a_u \to A_u$.
Note that $q^* q_! \yo a_u$ is the right fibration $q \downarrow qa_u \to A_u$, and the unit $\yo a_u \to q^* q_! \yo a_u$ is given fibrewise by the action $A_u(b, a_u) \to A_d(q b, q a_u)$ of $q$ on morphisms.

\begin{lemma}\label{gluing-pullback-properties}
Let $A_d, A_u, q, W_d, W_u, p, t_d, t_u$ be as above.
Suppose we are given a cartesian square of categories as below, where $C_d$ is a groupoid:
\[
\begin{tikzcd}
  C_u
  \ar[r, "k_u"]
  \ar[d, "r"']
  \ar[dr, phantom, "\lrcorner" very near start]
&
  A_u
  \ar[d, "q"]
\\
  C_d
  \ar[r, "k_d"]
&
  A_d \rlap{.}
\end{tikzcd}
\]
Then:
\begin{parts}
\item\label{gluing-pullback-properties:k-cart}
The functor $k_! \co \gl(r^*) \to \gl(q^*)$ preserves cartesian maps.
\item\label{gluing-pullback-properties:counit-yo-cart}
For any $a_u \co A_u$, the counit $k_! k^* \yogl a \to \yogl a$ is cartesian.
\item\label{gluing-pullback-properties:counit-t-cart}
For any $x \co \gl(p)$, the counit $k_! k^* t_! x \to t_! x$ is cartesian.
\end{parts}
\end{lemma}
\begin{proof}\leavevmode
\begin{parts}
\item
Let $x \to y$ be a cartesian map in $\gl(r^*)$.
We denote by $x_u \to C_u$ the right fibration $\pi_1x$, and similarly define $x_d, y_u, y_d$.
Thus the fact that $x \to y$ is cartesian means that $x_u \simeq x_d \times_{y_d} y_u$.
Using the pointwise formula for left Kan extension, we have to show that the following square of spaces is cartesian for $a_u : A_u$ with $a_d \coloneqq q a_u$:
\[
\begin{tikzcd}
  \verts{a_u \downarrow x_u}
  \ar[r]
  \ar[d]
&
  \verts{a_u \downarrow y_u}
  \ar[d]
\\
  \verts{a_d \downarrow x_d}
  \ar[r]
&
  \verts{a_d \downarrow y_d} \rlap{.}
\end{tikzcd}
\]
Since $C_d$ is a groupoid, so is $y_d$, and hence so is $a_d \downarrow y_d$.
So by \cref{localisation-pullback}, it suffices to show that the square above is a cartesian square of categories before localisation.
This follows from pullback pasting.
\item
Given $a_u, b_u : A_u$ with $a_d \coloneqq qa_u$ and $b_d \coloneqq q b_u$, we have to show that the following square of spaces is cartesian:
\[
\begin{tikzcd}
  \verts{b_u \downarrow C_u \downarrow a_u}
  \ar[r]
  \ar[d]
&
  A_u(b_u,a_u)
  \ar[d]
\\
  \verts{b_d \downarrow C_d \downarrow a_d}
  \ar[r]
&
  A_d(b_d,a_d) \rlap{.}
\end{tikzcd}
\]
Working fibrewise over $C_d$ and using that $C_u \simeq C_d \times_{A_d} A_u$, this amounts to a Conduch\'e condition.
\item
To describe $k_! k^* t_! x$, we use that $k^* t_!$ can be rewritten as restriction followed by left Kan extension (see \cref{subsection-lke}).
We thus have to show that for $a_u : A_u$ with $a_d \coloneqq qa_u$, the following square of spaces is cartesian:
\[
\begin{tikzcd}
  \verts{a_u \downarrow C_u \downarrow x_u}
  \ar[r]
  \ar[d]
&
  \verts{a_u \downarrow x_u}
  \ar[d]
\\
  a_d \downarrow C_d \downarrow x_d
  \ar[r]
&
  a_d \downarrow x_d \rlap{.}
\end{tikzcd}
\]
Using \cref{localisation-pullback} and working fibrewise over $C_d$, $x_d$, this reduces to a Conduch\'{e} condition.
\qedhere
\end{parts}
\end{proof}

\begin{lemma}[label=assumptions-satisfied]
With notation and assumptions as before, the diagram of categories below satisfies all the assumptions of \cref{S-good}.
Moreover, for any $a_u : A_u$, the object $\yogl a : \gl(q^*)$ is good.
\[
\begin{tikzcd}
  \rfib(A_u)
  \ar[r, "m_u^*"]
&
  \arr(\rfib(W_u))
&
  \rfib(W_u)
  \ar[l, "\Delta"', hook']
\\
  \rfib(A_d)
  \ar[r, "m_d^*"']
  \ar[u, "q^*"]
&
  \arr(\rfib(W_d))
  \ar[u, "\arr(p^*)"']
&
  \rfib(W_d)
  \ar[l, "\Delta", hook']
  \ar[u, "p^*"']
\end{tikzcd}
\]
\end{lemma}
\begin{proof}
It is immediate that both cospans (rows) satisfy the assumptions of \cref{Sinfty-reflector}, and that \cref{S-good:descent,S-good:preserve-colimits} of \cref{S-good} are satisfied.

We have $\gl(\arr(p^*)) \simeq \arr(\gl(p^*))$, with the left adjoint $\Delta_! \co \gl(\arr(p^*)) \to \arr(p^*)$ to $\Delta^*$ corresponding to $\cod \co \arr(\gl(p^*)) \to \gl(p^*)$.
A morphism in $\gl(\arr(p^*))$ is cartesian if it is sent to cartesian morphisms in $\gl(p^*)$ by both $\dom$ and $\cod$.
It follows that $\Delta_!$ preserves cartesian morphisms, \ie \cref{S-good:g-cart} of \cref{S-good} is satisfied.
We also have that, for $x : \gl(\arr(p^*))$, the unit $\eta_\Delta x \co x \to \Delta^* \Delta_! x$ is cartesian if and only if $x$ is cartesian as a morphism in $\gl(p^*)$.

The functor $m^* \co \gl(q^*) \to \gl(\arr(p^*))$ corresponds to the morphism $t^* \to s^*$ of functors $\gl(q^*) \to \gl(p^*)$.
By \cref{arr-left-adj}, the left adjoint $m_! \co \gl(\arr(p^*)) \to \gl(q^*)$ is given by the following pushout:
\[
\begin{tikzcd}
  s_! \dom
  \ar[r]
  \ar[d]
&
  s_! \cod
  \ar[d]
\\
  t_! \dom
  \ar[r]
&
  m_! \rlap{.}
  \ar[ul, phantom, "\ulcorner" very near start]
\end{tikzcd}
\]
We claim that the map $s_! x \to t_! x$ is cartesian for any $x : \gl(p^*)$.
This amounts to showing that the following square of spaces below is cartesian:
\[
\begin{tikzcd}
  \verts{a_u \downarrow_s x_u}
  \ar[r]
  \ar[d]
&
  \verts{a_u \downarrow_t x_u}
  \ar[d]
\\
  a_d \downarrow_s x_d
  \ar[r]
&
  a_u \downarrow_t x_u \rlap{.}
\end{tikzcd}
\]
This follows from \cref{localisation-pullback} and the fact that the components of $s_u \to t_u$ are $q$-cartesian.

We now argue that $m_!$ preserves every cartesian map $x \to y$ in $\gl(\arr(p^*))$ such that $y$ determines a cartesian map of $\gl(p^*)$.
Write $x_0 \to x_1$, $y_0 \to y_1$ for the maps in $\gl(p^*)$ determined by $x, y$.
The fact that $x \to y$ is cartesian means that $x_0 \to y_0$ and $x_1 \to y_1$ are cartesian; by composition and cancellation, so is $x_0 \to x_1$.
Now $m_! x \to m_! y$ is the map on pushouts induced by the map of spans below:
\[
\begin{tikzcd}
  t_! x_0
  \ar[d]
&
  s_! x_0
  \ar[d]
  \ar[l]
  \ar[r]
&
  s_! x_1
  \ar[d]
\\
  t_! y_0
&
  s_! y_0
  \ar[l]
  \ar[r]
&
  s_! y_1
  \end{tikzcd}
\]
Every map in the above diagram is cartesian either by \cref{gluing-pullback-properties:k-cart} of \cref{gluing-pullback-properties} or by the fact that $s_! \to t_!$ is always cartesian, so $m_! x \to m_! y$ is cartesian by \cref{big-list:cart-pushout} of \cref{big-list}.
This establishes \cref{S-good:f-cart} of \cref{S-good}.

It remains to verify that $\yogl a$ is good for all $a_u : A_u$ and that $m_! \Delta^* y$ is good for all $y : \gl(p^*)$.
Recall that an object $x : \gl(q^*)$ is good if $\varepsilon_m x$ and $t^* x \to s^* x$ are both cartesian.
Note that $m_! \Delta^* y$ is simply $t_! y$.
We have that $\varepsilon_m x$ is the cogap map of the following square.
\[
\begin{tikzcd}
  s_! t^* x
  \ar[r]
  \ar[d]
&
  s_! s^* x
  \ar[d]
\\
  t_! t^* x
  \ar[r]
&
  x \rlap{,}
\end{tikzcd}
\]
The map $s_! t^* x \to t_! t^* x$ is always cartesian, as shown above.
By \cref{big-list:cart-comp,big-list:cart-cogap} of \cref{big-list}, $\varepsilon_m x$ is cartesian if both counits $t_! t^* x \to x$ and $s_! s^* x \to x$ are cartesian.
In the cases at hand, where $x$ is either $\yogl a$ or $t_! c$, this follows from \cref{gluing-pullback-properties:counit-yo-cart,gluing-pullback-properties:counit-t-cart} of \cref{gluing-pullback-properties}

The map $t^* \yogl a \to s^* \yogl a$ being cartesian is a direct consequence of the fact that the components of $s_u \to t_u$ are $q$-cocartesian.

The map $t^* t_! y \to s^* t_! y$ being cartesian is the statement that the square
\[
\begin{tikzcd}
  t_u w_u \downarrow_t y_u
  \ar[r]
  \ar[d]
&
  s_u w_u \downarrow_t y_u
  \ar[d]
\\
  t_d w_d \downarrow_t y_d
  \ar[r]
&
  s_d w_d \downarrow_t y_d
\end{tikzcd}
\]
is cartesian, which again follows from $s_uw_u \to t_u w_u$ being $q$-cocartesian, together with \cref{localisation-pullback}.
\end{proof}

We are now finally ready to return to the main lemma.
Recall that we changed notation after stating the lemma, so that $E$ and $C$ in the statement of the lemma correspond to $A_u$ and $A_d$, respectively.

\begin{proof}[Proof of \cref{mapping-spaces-conduche}.]
We need to show both parts of \cref{mapping-spaces-cart}.
The square in \cref{mapping-spaces-cart:arrow} is precisely the square underlying the unit map $\yogl y \to S^\infty \yogl y$ evaluated at an object $x : A_u$.
This is cartesian by \cref{S-good,assumptions-satisfied}.

For \cref{mapping-spaces-cart:zig-zag-start}, suppose $y : A_u$ and $c : A_d$ and consider the cartesian square of categories%
\footnote{Here, $\braces{c}$ is the terminal category, and the functor $k_d \co \braces{c} \to A_d$ selects the object $c$.}
\[
\begin{tikzcd}
  \fib_q(c)
  \ar[r, "k_u"]
  \ar[d, "r"']
  \ar[dr, phantom, "\lrcorner" very near start]
&
  A_u
  \ar[d, "q"]
\\
  \braces{c}
  \ar[r, "k_d"]
&
  A_d \rlap{.}
\end{tikzcd}
\]
Say that an object $x : \gl(q^*)$ is \emph{nice} if the counit $k_! k^* x \to x$ is cartesian.
The desired statement is precisely the statement that $S^\infty \yogl y$ is nice.
Since $k_!k^*$ and $q^*$ preserve sequential colimits and sequential colimits
in $\rfib(A_u)$ commute with pullbacks, any sequential colimit of nice objects is nice.
We have that $\yogl y$ is nice by \cref{gluing-pullback-properties:counit-yo-cart} of \cref{gluing-pullback-properties}.
We claim that if $x : \gl(q^*)$ is nice \emph{and good}, then $Sx$ is also nice.
It then follows that $S^n \yogl y$ is nice and good for every $n$ (by \cref{S-good,assumptions-satisfied}), so that $S^\infty \yogl y$ is nice.

Since $k_! k^*$ preserves pushouts, in particular the pushout defining $Sx$, it suffices by \cref{big-list:cart-pushout} of \cref{big-list} to show that every labelled map in the below diagram is cartesian:
\[
\begin{tikzcd}[column sep=4em]
  k_! k^* m_! \Delta^* \Delta_! m^* x
  \ar[d, "\varepsilon_k \cdots"']
&
  k_! k^* m_! m^* x
  \ar[d, "\epsilon_k \cdots"]
  \ar[l]
  \ar[r, "k_! k^* \epsilon_m x"]
&
  k_! k^* x
  \ar[d, "\varepsilon_k x"]
\\
  m_! \Delta^* \Delta_! m^* x
&
  m_! m^* x
  \ar[l, "m_! \eta_\Delta m^* x"]
  \ar[r, "\varepsilon_m x"']
&
  x \rlap{.}
\end{tikzcd}
\]
The bottom maps are cartesian since $x$ is assumed to be good.
The right map is cartesian since $x$ is assumed to be nice.
To see that the left map is cartesian, note that $m_! \Delta^* \Delta_! m^* x \simeq t_! s^* x$ and apply \cref{gluing-pullback-properties:counit-t-cart} of \cref{gluing-pullback-properties}.
To see that $k_! k^* \varepsilon_m x$ is cartesian we apply \cref{gluing-pullback-properties:k-cart} of \cref{gluing-pullback-properties}.
\end{proof}

\subsection{Descent along localisations}

At this point, we treat \cref{mapping-spaces-cart} as a black box and make no further use of gluings or of the sequential-colimit description of mapping spaces of localisations.
Let us switch back to writing $C$, $E$, $W$ in place of $A_d$, $A_u$, $W_d$, respectively.

We record two easy observations about the localisation $i \co C \to C[W^{-1}]$.
First, it is surjective; indeed the inclusion of the essential image of $i$ into $C[W^{-1}]$ has a section by the universal property of $C[W^{-1}]$.
Second, given objects $x, y : C$ and a map $f \co i x \to i y$ in the localisation, there exists a sequence of objects $x = x_0, x_1, \ldots, x_n = y$ and maps $f_j \co i x_j \to i x_{j+1}$ whose composite is $f$ and such that each $f_j$ is either of the form $i g$ with $g \co x_j \to x_{j+1}$, or of the form $(i g)^{-1}$ with $g$ in $W$.
This can be shown by using that $i^* i_! C(-, y)$ is the initial presheaf under $C(-, y)$ with the property of inverting $W$.

\begin{lemma}[label=q-prime-cocart]
Let $q \co E \to C$ be a cocartesian fibration which inverts $W$ and let $W_u$ be the collection of $q$-cocartesian lifts of morphisms in $W$.
Let $q'$ be the induced functor $E[W_u^{-1}] \to C[W^{-1}]$.
Then $q'$ is a cocartesian fibration and the following square is cartesian:
\[
\begin{tikzcd}
  E
  \ar[r, "i_u"]
  \ar[d, "q"']
  \ar[dr, phantom, "\lrcorner" very near start]
&
  E[W_u^{-1}]
  \ar[d, "q'"]
\\
  C
  \ar[r, "i"]
&
  C[W^{-1}] \rlap{.}
\end{tikzcd}
\]
\end{lemma}
\begin{proof}
Let us denote the localisations $C[W^{-1}]$ and $E[W_u^{-1}]$ by $C'$ and $E'$, respectively.
By \cref{mapping-spaces-cart:arrow} of \cref{mapping-spaces-cart}, we have that the gap map $E \to E' \times_{C'} C$ is fully faithful.
In particular, this means that for $c : C$, the functor $\fib_q(c) \to \fib_{q'}(ic)$ is fully faithful.

We can reformulate \cref{mapping-spaces-cart:zig-zag-start} of \cref{mapping-spaces-cart} using the fact that $q$ is a cocartesian fibration.
Let $x, y : E$ and $c : C$.
Note that for $x : E$, the comma category $x \downarrow E_c$ has a coreflective subcategory spanned by $q$-cocartesian maps.
This means that for $f \co q x \to c$, the fibre of $\verts{x \downarrow E_c \downarrow i_u y}$ over $f$ is the mapping space $E'(i_u f_! x, i_u y)$.
In this way we get that the following square of spaces is cartesian:
\[
\begin{tikzcd}
  E'(i_u f_! x, i_u y)
  \ar[r]
  \ar[d]
  \ar[dr, phantom, "\lrcorner" very near start]
&
  E'(i_u x, i_u y)
  \ar[d]
\\
  C'(i c, i q y)
  \ar[r]
&
  C'(i q x, i q y) \rlap{.}
\end{tikzcd}
\]
Since every object of $E'$ is equivalent to one of the form $i_u y$, this means that $i_u g \co i_u x \to i_u f_! x$ is $q'$-cocartesian, where $g$ denotes the $q$-cocartesian map $x \to f_! x$.

We now claim that given $a, b : C$, $f : C'(i a, i b)$, and $x : \fib_{q'}(i a)$ such that $x$ lies in the image of $\fib_q(a) \hookrightarrow \fib_{q'}(ia)$, there is $y : \fib_{q'}(i b)$ and a $q'$-cocartesian lift $x \to y$ of $f$, and moreover, $y$ lies in the image of $\fib_q(b) \hookrightarrow \fib_{q'}(i b)$.
Since $q'$-cocartesian maps are closed under composition, it suffices to consider the cases where $f$ is either of the form $i g$, or of the form $(ig)^{-1}$ with $g$ in $W$.
The first case was dealt with above, so we consider the case of $f = (i g)^{-1}$.
Since $f$ is invertible in this case, there is trivially a $q'$-cocartesian lift, given by an isomorphism $x \simeq y$.
The fact that $y$ lies in the image of $\fib_q(b) \hookrightarrow \fib_{q'}(i b)$ can be seen by considering the square
\[
\begin{tikzcd}
  \fib_q(b)
  \ar[r, hook]
  \ar[d, "g_!"']
&
  \fib_{q'}(i b)
  \ar[d, "(i g)_!"]
\\
  \fib_q(a)
  \ar[r, hook]
&
  \fib_{q'}(i a) \rlap{,}
\end{tikzcd}
\]
in which the vertical maps are invertible, and $x : \fib_{q'}(i a)$ is the image of $y$ under $(i g)_!$.
Now let $a : C$ and $x : \fib_{q'}(i a)$ be arbitrary.
Since $i_u \co E \to E'$ is surjective, there is $y : E$ with $g \co i_u y \simeq x$.
In particular, $q' g$ gives us an isomorphism $f \co i q y \simeq i a$, and $g$ is precisely the $q'$-cocartesian lift of $f$.
By the claim above, this means that $x$ lies in the image of $\fib_q(a) \hookrightarrow \fib_{q'}(i a)$, so this functor is surjective.
This means that the gap map $E \to E' \times_{C'} C$ is surjective and hence invertible.
This also means that the claim above, together with surjectivity of $i \co C \to C'$, establishes that $q'$ is a cocartesian fibration.
\end{proof}

\begin{remark}
Conduch\'e fibrations enjoy a straightening--unstraightening correspondence which generalises that for (co)cartesian fibrations~\cite{ayala20}, and it is reasonable to ask if our methods apply also to Conduch\'e fibrations.
Indeed, the main ingredient, \cref{mapping-spaces-conduche}, applies to Conduch\'e fibrations.

However, straightening of Conduch\'e fibrations is inherently more complicated than straightening of cocartesian fibrations, since Conduch\'e fibrations are classified by \emph{flagged} functors (\ie, not just functors in the ordinary sense) into the flagged category of categories and profunctors between them.
Essentially this is because not every invertible profunctor is representable.

A concrete manifestation of this is that \cref{q-prime-cocart} does \emph{not} apply to Conduch\'e fibrations.
Indeed, let $F \co A \to B$ be some functor which defines an invertible profunctor but not an equivalence of categories. For example $A$ could be the walking idempotent
and $B$ the walking section-retraction pair.
We build a Conduch\'e fibration $q \co E \to \Delta^3$ as follows: the fibres over the objects $0, 1, 2, 3$ of $\Delta^3$ are $A$, $B$, $A$, $B$, respectively, and the profunctors between fibres induced by the generating morphisms $0 \to 1$, $1 \to 2$, $2 \to 3$ of $\Delta^3$ are given by $F$, $F^{-1}$, and $F$, respectively.
It can be seen that $q$ inverts the morphisms $0 \to 2$ and $1 \to 3$.
The localisation of $\Delta^3$ at these two morphisms is the terminal category $1$.
Since the fibres $A$ and $B$ over $0, 1 \co \Delta^3$ are not equivalent, it cannot be that $q \co E \to \Delta^3$ is the pullback of a Conduch\'e fibration $E' \to 1$.
\end{remark}

\begin{theorem}[label=descent-localisation]
Let $C$ be a category $W$ a collection of morphisms, and let $i \co C \to C[W^{-1}]$ be the corresponding localisation.
Then the (wild) functor $i^* \co \cocartfib(C[W^{-1}]) \to \cocartfib(C)$ given by base change along $i$ is fully faithful, and its image is spanned by cocartesian fibrations that invert $W$.
\end{theorem}
\begin{proof}
Since $i$ is surjective, $i^*$ is conservative.
Now let $q \co E \to C$ be a cocartesian fibration that inverts $W$.
By \cref{q-prime-cocart}, we obtain a cocartesian fibration $q' \co E[W_u^{-1}] \to C[W^{-1}]$ together with an isomorphism $\eta \co q \to i^* q'$.
We claim that $(q', \eta)$ is initial in $q \downarrow i^*$.
This means that for any other cocartesian fibration $p \co Y \to C[W^{-1}]$, the space of functors of cocartesian fibrations $q' \to p$ is equivalent to the space of functors of cocartesian fibrations $q \to i^*p$, by the map $F \mapsto i^*(F) \eta$.

By construction, the space $\map_{C[W^{-1}]}(q', p)$ of \emph{all} functors of categories over $C[W^{-1}]$, not necessarily preserving cocartesian maps, is equivalent to the subspace of
$\map_C(q, i^*p)$ spanned by those functors $F \co E \to i^* Y$ such that the composite $E \to Y$ inverts every morphism of $W_u$.
If a functor $F \co E \to i^* Y$ over $C$ preserves cocartesian maps, then so does $E \to Y$, so morphisms in $W_u$ are sent to cocartesian lifts of isomorphisms in $Y$, and so $E \to Y$ inverts $W_u$.
It remains only to show that the induced functor $F' \co E[W_u^{-1}] \to Y$ also preserves cocartesian maps.
This follows from the fact that every morphism in $C[W^{-1}]$ can be written as a composite of morphisms in the image of $i$ (whose cocartesian lifts are preserved by $F'$ since $F$ preserves cocartesian maps) and isomorphisms (whose cocartesian lifts are invertible and so preserved by every functor).

Thus, $q \mapsto q'$ determines a left adjoint to the functor from $\cocartfib(C[W^{-1}])$ to the full subcategory of $\cocartfib(C)$ spanned by cocartesian fibrations that invert $W$.
Since the unit of this adjunction is invertible and the right adjoint is conservative, this is an equivalence of (wild) categories.
\end{proof}

\section{Other descent statements}

In this section we prove descent statements for cocartesian fibrations over cocommas, sequential colimits, and groupoids.

\begin{lemma}[label=descent-cocomma]
Given a cocomma lax square of categories
\[
\begin{tikzcd}
  A
  \ar[r, "g"]
  \ar[d, "f"']
  \twocellflip{dr}
&
  C
  \ar[d, "k"]
\\
  B
  \ar[r, "h"']
&
  D \rlap{,}
\end{tikzcd}
\]
the corresponding lax square of (wild) categories of cocartesian fibrations
\[
\begin{tikzcd}
  \cocartfib(D)
  \ar[r, "h^*"]
  \ar[d, "k^*"']
  \twocell{dr}
&
  \cocartfib(B)
  \ar[d, "f^*"]
\\
  \cocartfib(C)
  \ar[r, "g^*"']
&
  \cocartfib(A)
\end{tikzcd}
\]
is a comma lax square.
\end{lemma}
\begin{proof}
Since $h$ and $k$ are jointly surjective, $h^*$ and $k^*$ are jointly conservative, and so the functor $R \co \cocartfib(D) \to f^* \downarrow g^*$ is conservative.
We claim that it has a left adjoint and that the unit is invertible; it then follows that it is an equivalence of (wild) categories.

An object of $f^* \downarrow g^*$ is given by cocartesian fibrations $p \co Y \to B$ and $q \co Z \to C$ together with a cocartesian functor $f^* Y \to g^* Z$ over $A$.
This data is described by the diagram
\[
\begin{tikzcd}
  Y
  \ar[d, "p"']
&
  f^* Y
  \ar[l, "u"']
  \ar[r, "v"]
  \ar[d]
  \ar[dl, phantom, "\llcorner" very near start]
&
  Z
  \ar[d, "q"]
\\
  B
&
  A
  \ar[l, "f"]
  \ar[r, "g"']
&
  C \rlap{.}
\end{tikzcd}
\]
Let $W$ denote the cocomma of $Y \xleftarrow{u} f^*Y \xrightarrow{v} Z$, and denote the inclusions into $W$ by $s \co Y \hookrightarrow W$ and $t \co Z \hookrightarrow W$.
We claim that the induced functor $r \co W \to D$ is a cocartesian fibration.

There are three kinds of morphisms in $D$: those coming from $B$, those coming from $C$, and ``heteromorphisms'', of the form $hb \to kc$.
The base changes of $r$ along $h$ and $k$ are $p$ and $q$, respectively, essentially because $h$ and $k$ are fully faithful.
This shows that $r$ has cocartesian lifts of morphisms coming from $C$, with $t \co Z \hookrightarrow W$ sending $q$-cocartesian maps to $r$-cocartesian maps.
It also shows that $r$ has \emph{locally cocartesian} lifts of morphisms coming from $B$.
It remains to show that it has cocartesian lifts of heteromorphisms $h c \to k c$, and that these are stable under pre- and postcomposition with morphisms coming from $B$ or $C$.

Consider objects $y : Y$ and $z : Z$.
Then the mapping space $W(s y, t z)$ is the localisation of $y \downarrow f^* Y \downarrow z$, and the mapping space $D(h p y, k q z)$ is the localisation of $p q \downarrow A \downarrow q z$.
Consider the coreflective subcategory $y \downarrow^\cocart f^*Y \downarrow Z$ of $y \downarrow f^*Y \downarrow z$ spanned by objects whose underlying morphism in $Y$ is $p$-cocartesian.
Then we can equivalently compute $D(h p y, k q z)$ as the localisation of $y \downarrow^\cocart f^* Y \downarrow Z$.

Now consider the functor $F \co (y \downarrow^\cocart f^*Y \downarrow Z) \to (pq \downarrow A \downarrow qz)$.
We claim that $F$ is a \emph{Kan fibration}, \ie, that it is both a left fibration and a right fibration.
The fibre of $F$ over an object given by $a : A$ with $m \co p q \to f a$ and $n \co g a \to q z$ is the mapping space $Z_{q z}(n_! v(m_! y, a), z)$.
Showing that $F$ is a Kan fibration boils down to the fact that, given $a, a' : A$ with maps $m \co p q \to f a$ and $l \co a \to a'$ and $n \co g a' \to q z$, we have $(n g(l))_! v(m_! y, a) \simeq n_! v((f(l) m)_! y, a')$, since $v$ sends the cocartesian map $(m_! y, a) \to ((f(l) m)_! y, a')$ to a $q$-cocartesian lift of $g(l)$.

Admitting that $F$ is a Kan fibration, descent implies that the square of categories
\[
\begin{tikzcd}
  y \downarrow^\cocart f^*Y \downarrow z
  \ar[r]
  \ar[d]
&
  W(s y, t z)
  \ar[d]
\\
  p y \downarrow A \downarrow q z
  \ar[r]
&
  D(h p y, k q z) \rlap{.}
\end{tikzcd}
\]
is cartesian.
Since the bottom map is surjective (indeed, a localisation), this means that every fibre of $W(s y, t z) \to D(h p y, k q z)$ is of the form $Z_{q z}(n_! v(m_! y , a), z)$.
In particular, heteromorphisms have cocartesian lifts: these correspond to isomorphisms $n_! v(m_!y,a) \simeq n_! v(m_! y, a)$.
From this description, it is clear that they are stable under precomposition with locally cocartesian lifts of morphisms from $B$.
This finishes the proof that $r$ is a cocartesian fibration.

The inclusions into the cocomma $W$ define a map from $Y \xleftarrow{u} f^*Y \xrightarrow{v} Z$ to $Rr$ in $f^* \downarrow g^*$.
Moreover, this map is invertible, since $Y \simeq h^* W$ and $Z \simeq k^* W$.
Moreover, this defines left adjoint to $R$, essentially by the universal property of $W$ as a cocomma.
Since $R$ is conservative and has a left adjoint with invertible unit, $R$ is invertible.
\end{proof}

\begin{lemma}[label=descent-sequential]
Given a sequential colimit diagram of categories
\[
\begin{tikzcd}
  C_0
  \ar[r, "f_0"]
  \ar[drr, "i_0"']
&
  C_1
  \ar[r, "f_1"]
  \ar[dr, "i_1"]
&
  \cdots
\\&&
  C_\infty \rlap{,}
\end{tikzcd}
\]
the corresponding diagram of (wild) categories of cocartesian fibrations
\[
\begin{tikzcd}
  \cocartfib(C_\infty)
  \ar[dr, "i_1^*"']
  \ar[drr, "i_0^*"]
\\
  \cdots
  \ar[r]
&
  \cocartfib(C_1)
  \ar[r]
&
  \cocartfib(C_0)
\end{tikzcd}
\]
is a sequential limit diagram.
\end{lemma}
\begin{proof}
Since $\II$ detects equivalences of categories and $\map(\II, -)$ commutes with sequential colimits, sequential colimits of categories satisfy descent, \ie, $\cat / C_\infty$ is the sequential limit of $\cat / C_0 \xleftarrow{f_0^*} \cat / C_1 \cdots$.
Again using compactness of $\II$ and the fact that sequential colimits commute with pullback in spaces, one can check that a category over $C_\infty$ is a cocartesian fibration as soon as its pullback to $C_n$ is a cocartesian fibration for every $n$, and similarly with functors of cocartesian fibrations over $C_\infty$.
\end{proof}

\begin{lemma}[label=descent-groupoid]
Given a groupoid $X$, the (wild) functor
\[
  \cocartfib(X) \to \cat^X
\]
which computes the fibres of a cocartesian fibration is an equivalence of (wild)
categories.
\end{lemma}
\begin{proof}
Every category over $X$ is a cocartesian fibration, with cocartesian maps given by the isomorphisms.
Therefore, $\cocartfib(X)$ is just the (wild) category of categories over $X$.
The claim thus reduces to extensivity of the (wild) category of categories.
\end{proof}

\section{A directed join construction}

In this section, we describe a model-dependent way of decomposing a general category into simpler pieces.
More precisely, for any functor $f \co A \to B$, we give a sequential colimit formula for the full image of $f$.
This is a directed generalisation of the join construction~\cite{rijke17}, which deals with the case where $A$ and $B$ are groupoids.
We are primarily interested in the case where $f$ is a \emph{core inclusion} $C^\simeq \to C$.

\begin{construction}[label=directed-join]
Consider categories $A_0$, $A_1$, $B$ with functors $f_0 \co A_0 \to B$ and $f_1 \co A_1 \to B$.
We define $A_0 \vec{*}_B A_1$ to be the category freely generated by functors $i_0 \co A_0 \to A_0 \vec{*}_B A_1$ and $i_1 \co A_1 \to A_0 \vec{*}_B A_1$ together with a natural transformation $\alpha \co i_0 \dom \to i_1 \cod$ of functors $f_0 \downarrow f_1 \to A_0 \vec{*}_B A_1$, subject to the condition that the restriction of $\alpha$ along $A_0 \times_B A_1 \hookrightarrow f_0 \downarrow f_1$ is invertible.
This category comes with a functor $A_0 \vec{*}_B A_1 \to B$, which we denote by $i_0 \vec{*} i_1$:
\[
\begin{tikzcd}
  A_0 \times_B A_1
  \ar[r, hook]
&
  f_0 \downarrow f_1
  \ar[r, "\cod"]
  \ar[d, "\dom"']
  \twocellflip{dr}["\alpha", pos=0.3]
&
  A_1
  \ar[d, "i_1"]
\\&
  A_0
  \ar[r, "i_0"']
&
  A_0 \vec{*}_B A_1
  \ar[r, "i_0 \vec{*} i_1"]
&
  B \rlap{.}
\end{tikzcd}
\]
\end{construction}
Explicitly, we can build $A_0 \vec{*}_B A_1$ as the localisation of the cocomma of $A_0 \xleftarrow{\dom} f_0 \downarrow f_1 \xrightarrow{\cod} A_1$ at the collection of morphisms given by applying the generating 2-cell to an object of $f_0 \downarrow f_1$ whose underlying morphism in $B$ is invertible.

\begin{lemma}[label=join-correctness]
Let $A$, $B$, $X_0$ be categories and let $f \co A \to B$ and $g_0 \co X_0 \to B$ be functors.
For every $n : \mathbb N$, a category $g_n \co X_n \to B$ over $B$ is given by taking $X_{n+1} \coloneqq A \vec{*}_B X_n$, $g_{n+1} \coloneqq f \vec{*} g_n$.
Let $X_\infty$ be the colimit of the sequence $X_0 \xrightarrow{i_1} X_1 \xrightarrow{i_1} \cdots$.
Then assuming the image of $g_0$ is contained in the image of $f$, we have that $X_\infty \xrightarrow{g_\infty} B$ is fully faithful, and exhibits $X_\infty$ as the full image of $f$.
\end{lemma}

\begin{proof}
For $0 \leq m \leq n \leq \infty$, denote the functor $X_m \to X_n$ by $\iota_m^n$.
For $n : \mathbb N$, denote the functor $A \to X_n$ by $j_n$.
Since the functors $j_n \co A \to X_{n+1}$ and $\iota_n^{n+1} \co X_n \to X_{n+1}$ are jointly surjective, the image of $g_{n+1}$ is the union of the images of $f$ and $g_n$.
It follows that the image of $g_n$ is contained in the image of $f$ for every $n$.
Since the functors $\iota_n^\infty \co X_n \to X_\infty$ are jointly surjective, the image of $g_\infty$ coincides with the image of $f$.
Thus it remains to show that $g_\infty$ is fully faithful.

We want to show that, for every pair of objects $x', y : X_\infty$, the map $X_\infty(x',y) \to B(g_\infty x', g_\infty y)$ is an equivalence of spaces.
Since the functors $\iota_n^\infty \co X_n \to X_\infty$ are jointly surjective, we may assume that $x'$ is of the form $\iota_n^\infty x$ with $x : X_n$.
Without loss of generality, we may assume $n = 0$.
We may also assume we are given $a : A$ and an isomorphism $e \co f a \simeq g_0 x$ in $B$.

Since $X_1$ is defined as $A \vec{*}_B X_0$, we have a lax square
\[
\begin{tikzcd}
  f \downarrow g_0
  \ar[r, "\cod"]
  \ar[d, "\dom"']
  \twocellflip{dr}["\alpha", pos=0.3]
&
  X_0
  \ar[d, "\iota_0^1"]
\\
  A
  \ar[r, "j_0"']
&
  X_1 \rlap{.}
\end{tikzcd}
\]
This induces a functor $f \downarrow g_0 \to j_0 \downarrow \iota_0^1$.
Restricting along $a \co 1 \to A$, we get a functor $H \co f a \downarrow g_0 \to j_0 a \downarrow \iota_0^1$ of left fibrations over $X_0$.
By construction of $X_1$, we have that $H$ sends $e \co fa \xrightarrow{\sim} g_0 x$ to an isomorphism $j_0 a \simeq \iota_0^1 x$.
We also have a functor $G \co j_0 a \downarrow \iota_0^1 \to g_1 j_0 a \downarrow g_1 \iota_0^1$, given essentially by the action of $g_1$ on morphism.
The composite $G H \co f a \downarrow g_0 \to g_1 j_0 a \downarrow g_1 \iota_0^1$ is the equivalence induced by $f \simeq g_1 j_0$ and $g_0 \simeq g_1 \iota_0^1$.
We can now build the commutative diagram in $\lfib(X_0)$:
\begin{equation}\label{lfib-diagram}
\begin{tikzcd}
  f a \downarrow g_0
  \ar[r, "H"]
  \ar[d, "\sim"']
&
  j_0 a \downarrow \iota_0^1
  \ar[r, "G"]
  \ar[d, "\sim"]
&
  g_1 j_0 a \downarrow g_1 \iota_0^1
  \ar[d, "\sim"]
\\
  g_0 x \downarrow g_0
  \ar[r, "H'"']
&
  \iota_0^1 x \downarrow \iota_0^1
  \ar[r, "G'"]
&
  g_1 \iota_0^1 x \downarrow g_1 \iota_0^1 \rlap{.}
\end{tikzcd}
\end{equation}
The vertical equivalences are induced by $e \co f a \simeq g_0 x$ and $H e \co j_0 a \simeq \iota_0^1 x$.
The functor $G'$ is given by the action of $g_1$ on morphisms, and the composite $G' H'$ by the isomorphism $g_0 \simeq g_1 \iota_0^1$.
We can now forget about $a$ and the top row of the diagram.

Since $G'$ is the base change of the functor $\iota_0^1 x \downarrow X_1 \to g_1 \iota_0^1 x \downarrow g_1$ along $\iota_0^1$, we have also the following square:
\[
\begin{tikzcd}
  x \downarrow X_0
  \ar[r]
  \ar[d]
&
  \iota_0^1 x \downarrow X_1
  \ar[d]
\\
  g_0 x \downarrow g_0
  \ar[r]
  \ar[ur, dashed]
&
  g_1 \iota_0^1 x \downarrow g_1 \rlap{.}
\end{tikzcd}
\]
All sides admit a simple description not involving $a$ or $e$: the vertical functors are given by the actions of $g_0$ and $g_1$ on morphisms, the top map is induced by the action of $\iota_0^1$ on morphisms, and the bottom map is induced by the isomorphism $g_0 \simeq g_1 \iota_0^1$.
The indicated diagonal functor is induced by $H'$.
The bottom right triangle commutes since it is induced by $G' H'$, which we discussed above, and the top left triangle can be seen to commute by the Yoneda lemma.

Repeating this construction for every $n : \mathbb N$, we obtain the following commutative diagram:
\[
\begin{tikzcd}
  x \downarrow X_0
  \ar[r]
  \ar[d]
&
  \iota_0^1 x \downarrow X_1
  \ar[r]
  \ar[d]
&
  \iota_0^2 x \downarrow X_2
  \ar[r]
  \ar[d]
&
  \cdots
\\
  g_0 x \downarrow g_0
  \ar[r]
  \ar[ur]
&
  g_1 \iota_0^1 x \downarrow g_1
  \ar[r]
  \ar[ur]
&
  g_2 \iota_0^2 x \downarrow g_2
  \ar[r]
  \ar[ur]
&
  \cdots \rlap{.}
\end{tikzcd}
\]
The colimit of the top row is $\iota_0^\infty x \downarrow X_\infty$. The colimit of the bottom row is $g_\infty \iota_0^\infty x \downarrow g_\infty$.
The diagonal maps tell us that the induced functor $F \co \iota_0^\infty x \downarrow X_\infty \to g_\infty \iota_0^\infty x \downarrow g_\infty$ is invertible.
It remains to show that $F$ is isomorphic to the functor given by the action of $g_\infty$ on morphisms; since $F$ is invertible it follows from this that $g_\infty$ is fully faithful.

By the Yoneda lemma, it suffices to show that $F$ defines a morphism of categories over $X_\infty$, and that it sends the identity on $\iota_0^\infty x$ to the identity on $g_\infty \iota_0^\infty x$.
The former amounts to the fact that the witnesses of commutativity in \eqref{lfib-diagram} witness commutativity in $\lfib(X_0)$ and not just in $\cat$.
The latter follows from the fact that the functor $x \downarrow X_0 \to g_0 x \downarrow g_0$ sends the identity on $x$ to the identity on $g_0 x$.
\end{proof}

The directed join construction gives the following induction principle, which allows us to prove a result for a general category ``by induction'' on how the category is built from simpler pieces.

\begin{lemma}[label=join-induction]
Let $P$ be a class of categories such that:
\begin{conditions}
\item
Every groupoid is in $P$.
\item
If $B \xleftarrow{f} A \xrightarrow{g} C$ is a span of categories such that
$A$, $B$, $C$ are all in $P$, then $\cocomma{f}{g}$ is also in $P$.
\item
If $C_0 \to C_1 \to \cdots$ is a sequence of categories that are all in $P$, then their colimit is also in $P$.
\item
If $C$ is in $P$ and $D$ is a localisation of $C$ at some collection of morphisms,
then $D$ is in $P$.
\end{conditions}
Then every category is in $P$.
\end{lemma}
\begin{proof}
Let $C$ be a category; we want to show that $C$ is in $P$.
Consider \cref{join-correctness} with $A = X_0 = C^\simeq$ and $B = C$.
Since the full image of the core inclusion $C^\simeq \to C$ is all of $C$, this exhibits $C$ as the sequential colimit of a sequence of categories
$X_0 \to X_1 \to \cdots$.
It thus suffices to show that $X_n$ is in $P$ for every $n$.

Let $Q$ be the class of categories $Y$ such that for every left fibration $Z \to Y$, the category $Z$ lies in $P$.
We will show, by induction on $n$, that $X_n$ lies in $Q$; it follows that $X_n$ lies in $P$ since the identity $X_n \to X_n$ is a left fibration.

Note that if $Y$ is a groupoid, then $Y$ lies in $Q$, since for $Z \to Y$ a left fibration, $Z$ is also a groupoid.

We claim that if $Y$ lies in $Q$ and $Y'$ is a localisation of $Y$, then $Y'$ also lies in $Q$.
Indeed, if $Z \to Y'$ is a left fibration, then $Z \times_{Y'} Y$ lies in $P$ since $Y$ lies in $Q$, and $Z$ is a localisation of $Z \times_{Y'} Y$, so it also lies in $P$.

Next we claim that for any span of categories $B \xleftarrow{f} A \xrightarrow{g} C$ where $B$ is a groupoid, $g$ is a left fibration, and $C$ lies in $Q$, the cocomma category $\cocomma{f}{g}$ also lies in $Q$.
Suppose $W \to \cocomma{f}{g}$ is a left fibration.
Denote by $Y \to B$ and $Z \to C$ its base changes along $Y \to \cocomma{f}{g}$ and $Z \to \cocomma{f}{g}$, respectively.
Then $W$ is the cocomma of $Y \leftarrow f^* Y \to Z$.
$Y$ lies in $P$ since it is a groupoid.
$Z$ lies in $P$ since $Z \to C$ is a left fibration, and $f^* Y$ lies in $P$ since $f^*Y \to A$ and $g \co A \to C$ are left fibrations.
Thus $W$ lies in $P$, as needed.

It is now direct by induction that $X_n$ lies in $Q$ for every $n$.
\end{proof}

\subsection{A virtual directed join construction}

To construct directed univalent cocartesian fibrations, we now want to replay the directed join construction for cocartesian fibrations $E \to C$ (thought of as functors $C \to \cat$).
The analogue of comma categories turns out to be a bit subtle.

\begin{lemma}
Let $p_0 \co E_0 \to C_0$ and $p_1 \co E_1 \to C_1$ be two cocartesian fibrations.
Assuming that $C_0$ is a groupoid, then there is a category $\funcocart(p_0,p_1)$ which is cofreely generated by functors
\begin{align*}
u_0 &: \funcocart(p_0, p_1) \to C_0,\\
u_1 &: \funcocart(p_0, p_1) \to C_1
\end{align*}
with a functor $u_{01} \co u_0^* E_0 \to u_1^* E_1$ of cocartesian fibrations over $F$.
\end{lemma}

In the end, such a category can be shown to exist even if $C_0$ is not a groupoid (\eg, by taking the comma category of the straightenings of $p_0$ and $p_1$), but for now we will make do with the restricted statement.

\begin{proof}
We assume that $C_0$ is the terminal category; the general case follows by extensivity over $C_0$.
Define a category $F'$ by the pullback square
\[
\begin{tikzcd}
  F'
  \ar[r]
  \ar[d]
  \ar[dr, phantom, "\lrcorner" very near start]
&
  \fun(E_0, E_1)
  \ar[d]
\\
  C_1
  \ar[r, "\Delta"']
&
  \fun(E_0, C_1)
\end{tikzcd}
\]
Then $F'$ is cofreely generated by functors $u_0 \co F' \to C_0$ (this data is trivial) and $u_1 \co F' \to C_1$ together with a functor $u_{01} \co u_0^* E_0 \to u_1^* E_1$ of categories over $F'$.
In particular, an object of $F'$ is given by an object $c : C_1$ together with a functor $a \co E_0 \to \fib_{p_1}(c)$.
A morphism from $(c, a)$ to $(d, b)$ is given by a morphism $f \co c \to d$ together with a natural transformation $f_! a \to b$ of functors $E_0 \to \fib_{p_1}(d)$.
Let $F$ denote the wide subcategory of $F'$ spanned by morphisms where the natural transformation $f_! a \to b$ is invertible.
We claim that a functor $h \co D \to F'$ factors through $F$ if and only if the functor $h^* u_{01} \co h^* u_0^* E_0 \to h^* u_1^* E_1$ of categories over $D$ preserves cocartesian maps; it then follows that $F$ has the desired universal property.

Since $C_0$ is terminal, $h^*u_0^* E_0$ is simply $E_0 \times D$, and a morphism is cocartesian for the projection $E_0 \times D \to D$ if and only if it is of the form $(\id_e, g) \co (e, x) \to (e, y)$ for $g : D(x, y)$.
The functor $h \co D \to F'$ assigns to $x$ and $y$ respective functors $a \co E_0 \to \fib_{p_1}(u_1 h x)$ and $b \co E_0 \to \fib_{p_1}(u_1 h y)$, and to $g$ a natural transformation $\beta \co (u_1 h g)_! a \to b$.
The functor $h^* u_{01}$ sends $(\id_e, g)$ to a morphism in $E_1$ dgiven by the composite $a e \xrightarrow{\cocart} (u_1 g)_! a e \xrightarrow{\beta e} b e$.
This is $p_1$-cocartesian if and only if $\beta e$ is invertible.
Thus $h^* u_{01}$ preserves every cocartesian morphism of the form $(\id_e, f)$ if and only if $hf$ lies in $F'$.
\end{proof}

The following construction should be thought of as a virtual analogue of \cref{directed-join}.

\begin{construction}[label=virtual-join]
Given cocartesian fibrations $p_0 \co E_0 \to C_0$ and $p_1 \co E_1 \to C_1$ where $C_0$ is a groupoid, we denote by $C_0 \vec{*}_{p_0, p_1} C_1$ the localisation of the cocomma of $C_0 \xleftarrow{u_0} \funcocart(p_0, p_1) \xrightarrow{u_1} C_1$ at the collection of morphisms corresponding to objects $x$ of $\funcocart(p_0, p_1)$ whose underlying functor $\fib_{p_0}(u_0 x) \to \fib_{p_1}(u_1 x)$ is invertible.

Denote the inclusions from $C_0$, $C_1$ into $C_0 \vec{*}_{p_0,p_1} C_1$ by $i_0$, $i_1$.
Denote the natural transformation $i_0 u_0 \to i_1 u_1$ by $\alpha$.

By \cref{descent-localisation,descent-cocomma}, we obtain a cocartesian fibration $p_0 \vec{*} p_1$ over $C_0 \vec{*}_{p_0,p_1} C_1$ with witnesses that its base changes along $i_0$, $i_1$ are $p_0$, $p_1$, as well as a witness that the cocartesian functor $u_{01} \co u_0^* p_0 \to u_1^* p_1$ is induced by $\alpha \co i_0 u_0 \to i_1 u_1$.
\end{construction}

Recall that a functor $p \co E \to B$ is said to \emph{classify} a category $C$ if there exists an object $b : B$ together with an equivalence of categories $C \simeq \fib_p(b)$.
A cocartesian fibration $p \co E \to B$ is said to be \emph{directed univalent} if for every $x, y : B$, the map $B(x, y) \to \map(\fib_p(x), \fib_p(y))$ given by $f \mapsto f_!$ is an equivalence of spaces.

The following lemma can be thought of as a virtual analogue of \cref{join-correctness}, or as a \emph{directed} analogue of Uemura's construction of univalent completion~\cite[Proposition 4.7]{uemura25}.

\begin{lemma}[label=virtual-join-correctness]
Consider cocartesian fibrations $p \co B \to A$ and $q_0 \co Y_0 \to X_0$ where $A$ is a groupoid.
Define inductively over $n : \mathbb N$ a cocartesian fibration $q_n \co Y_n \to X_n$ by taking $q_{n+1}$ to be the cocartesian fibration $p \vec{*} q_n$ of \cref{virtual-join}.
Let $X_\infty$ denote the colimit of the sequence $X_0 \xrightarrow{i_1} X_1 \xrightarrow{i_1} X_2 \to \cdots$.
By \cref{descent-sequential}, we obtain a cocartesian fibration $q_{\infty} \co Y_\infty \to X_\infty$ whose base change along $X_n \to X_\infty$ is $q_n$.

Suppose that every category classified by $q_0$ is also classified by $p$.
Then $q_\infty$ is directed univalent, and a category is classified by $q_\infty$ if and only if it is classified by $p$.
\end{lemma}

\begin{proof}
We follow the proof of \cref{join-correctness} mutatis mutandis.
Denote the functor $A \to X_{n+1}$ by $j_n$, and the functor $X_m \to X_n$ for $0 \le m \le n \le \infty$ by $\iota_m^n$.
We have that $j_n$ and $\iota_n^{n+1}$ are jointly surjective, so every category classified by $q_n$ is also classified by $p$.
Since the functors $\iota_n^\infty$ are jointly surjective, a category is classified by $q_\infty$ if and only if it is classified by
$p$.
It remains to show that $q_\infty$ is directed univalent, \ie, that for all $x', y : X_\infty$, the map $X_\infty(x',y) \to \map(\fib_{q_\infty}(x'),\fib_{q_\infty}(y))$
is invertible.
Since the functors $\iota_n^\infty$ are jointly surjective, it suffices to consider $x'$ of the form $\iota_n^\infty x$ with $x : X_n$.
Without loss of generality we may assume $n = 0$.
We may also assume that we are given $a : A$ and an equivalence $e \co \fib_p(a) \simeq \fib_{q_0}(x)$.

Since $X_1$ is defined as $A \vec{*}_{p,q_0} X_0$, we have a lax square
\[
\begin{tikzcd}
  \funcocart(p, q_0)
  \ar[r, "u_1"]
  \ar[d, "u_0"']
  \twocellflip{dr}["\alpha", pos=0.3]
&
  X_0
  \ar[d, "\iota_0^1"]
\\
  A
  \ar[r, "j_0"']
&
  X_1 \rlap{.}
\end{tikzcd}
\]
Restricting along $a \co 1 \to A$, this gives a functor $H \co \funcocart(a^* p, q_0) \to j_0 a \downarrow \iota_0^1$ of left fibrations over $X_0$.
By construction of $X_1$, we have that $H$ sends $e$ to an isomorphism $j_0 a \simeq \iota_0^1 x$.
We also have a functor $G \co j_0 a \downarrow \iota_0^1 \to \funcocart((j_0 a)^* q_1 ,\iota_0^{1*} q_1)$ of left fibrations over $X_0$, which sends a morphism $f \co j_0 a \to \iota_0^1 y$ to the transport functor $f_! \co \fib_{q_1}(j_0 a) \to \fib_{q_1}(\iota_0^1 y)$.
The composite $G H \co \funcocart(a^*p, q_0) \to \funcocart((j_0 a)^* q_1, \iota_0^{1*} q_1)$ is the equivalence induced by $p \simeq j_0^* q_1$ and $q_0 \simeq \iota_0^{1*} q_1$.
We can now build the following commutative diagram in $\lfib(X_0)$:
\[
\begin{tikzcd}
  \funcocart(a^*p, q_0)
  \ar[r, "H"]
  \ar[d, "\sim"']
&
  j_0 a \downarrow \iota_0^1
  \ar[r, "G"]
  \ar[d, "\sim"]
&
  \funcocart((j_0 a)^* q_1,\iota_0^{1*} q_1)
  \ar[d, "\sim"]
\\
  \funcocart(x^* q_0, q_0)
  \ar[r, "H'"']
&
  \iota_0^1 x \downarrow \iota_0^1
  \ar[r, "G'"']
&
  \funcocart((\iota_0^1 x)^* q_1, \iota_0^{1*} q_1) \rlap{.}
\end{tikzcd}
\]
The vertical equivalences are induced by $e \co a^* p \simeq x^* q_0$ and $H e \co j_0 a \simeq \iota_0^1 x$.
The functor $G'$ sends a morphism $f$ to the transport functor $f_!$.
The composite $G' H'$ is induced by the equivalence $q_0 \simeq \iota_0^1 q_1$.
We can now forget about $a$ and the top row of the diagram above.

Since $G'$ is the base change of $\iota_0^1 x \downarrow X_1 \to \funcocart((\iota_0^1 x)^* q_1, q_1)$ along $\iota_0^1 \co X_0 \to X_1$, we also have the following square:
\begin{equation}\label{lfib-diagram-2}
\begin{tikzcd}
  x \downarrow X_0
  \ar[r]
  \ar[d]
&
  \iota_0^1 x \downarrow X_1
  \ar[d]
\\
  \funcocart(x^* q_0, q_0)
  \ar[r]
  \ar[ur, dashed]
&
  \funcocart((\iota_0^1 x)^* q_1, q_1) \rlap{.}
\end{tikzcd}
\end{equation}
All sides admit a simple description not involving $a$ or $e$.
The indicated diagonal functor is induced by $H'$.
The bottom right triangle commutes since it is induced by $G' H'$, which we discussed above.
Top left triangle can be seen to commute by the Yoneda lemma.

Repeating this construction for every $n: \mathbb N$, we obtain the following commutative diagram:
\[
\begin{tikzcd}
  x \downarrow X_0
  \ar[r]
  \ar[d]
&
  \iota_0^1 x \downarrow X_1
  \ar[r]
  \ar[d]
&
  \iota_0^2 x \downarrow X_2
  \ar[d]
  \ar[r]
&
  \cdots
\\
  \funcocart(x^* q_0, q_0)
  \ar[r]
  \ar[ur]
&
  \funcocart((\iota_0^1 x)^* q_1, q_1)
  \ar[r]
  \ar[ur]
&
  \funcocart((\iota_0^2 x)^* q_2, q_2)
  \ar[r]
  \ar[ur]
&
  \cdots \rlap{.}
\end{tikzcd}
\]
The colimit of the top row is $\iota_0^\infty x \downarrow X_\infty$. The colimit of the bottom row is $\funcocart((\iota_0^\infty x)^* q_\infty, q_\infty)$.
The diagonal maps tell us that the induced map $F \co \iota_0^\infty x \downarrow X_\infty \to \funcocart((\iota_0^\infty x)^* q_\infty, q_\infty)$ is invertible.
It remains to show that $F$ is isomorphic to the functor $f \mapsto f_!$; since $F$ is invertible it follows from this that $q_\infty$ is directed univalent.

By the Yoneda lemma, it suffices to show that $F$ defines a morphism of categories over $X_\infty$, and that it sends the identity on $\iota_0^\infty x$ to the identity functor on $(\iota_0^\infty x)^* q_\infty$.
The former amounts to the fact that the witnesses of commutativity in \eqref{lfib-diagram-2} witness commutativity in $\lfib(X_0)$ and not just in $\cat$.
The latter follows from the fact that the functor $(x \downarrow X_0) \to \funcocart(x^*q_0, q_0)$ sends the identity on $x$ to the identity functor on $x^* q_0$.
\end{proof}

We can now easily prove the existence of enough directed univalent cocartesian fibrations.

\begin{theorem}[label=enough,store=enough]
Given a functor $q \co Y \to X$, there is a directed univalent cocartesian fibration $p \co E \to B$, such that $p$ classifies every fibre of $q$.
\end{theorem}

\begin{proof}
Let $q \co Y \to X$ be an arbitrary functor, and let $P$ be the collection of categories classified by $q$.
Equivalently, $P$ is the collection of categories classified by $q' \co Y \times_X X^\simeq \to X^\simeq$.
By \cref{virtual-join-correctness}, we obtain a directed univalent cocartesian fibration which classifies the same categories.
\end{proof}

\section{Straightening of cocartesian fibrations}

Finally we prove a straightening theorem against any directed univalent fibration.

\begin{theorem}[label=straightening,store=straightening]
Suppose $p \co E \to B$ is a directed univalent cocartesian fibration and $C$ a category.
Given a cocartesian fibration $q \co D \to C$ such that $p$ classifies every fibre of $q$, there is a functor $f \co C \to B$ and an equivalence $D \simeq C \times_B E$ of categories over $C$.
Moreover, given two functors $f, g \co C \to B$, the canonical map from $\map(f, g)$ to the space of cocartesian functors $f^* p \to g^* p$ is invertible.
\end{theorem}


A direct proof of the result above is possible if one assumes that cocartesian fibrations are exponentiable~\cite{cisinski25}.
We instead deduce it from descent.

\begin{proof}
For a category $C$, denote the full subcategory of $\cocartfib(C)$ spanned by cocartesian fibrations all of whose fibres are classified by $p$ by $\cocartfib_p(C)$.
Say a category $C$ is \emph{good} if the (wild) functor $(-)_C^*p \co \fun(C, B) \to \cocartfib_p(C)$ is invertible.
We have to show that every category is good.
By \cref{join-induction}, it suffices to show that every groupoid is good, and that good categories are closed under cocommas, localisations, and sequential colimits.

Since $p$ is directed univalent, in particular the map $B^\simeq \to \cat^\simeq$ which computes fibres of $p$ is a monomorphism.
Together with \cref{descent-groupoid}, this means that every groupoid is good.

Suppose we are given a span of categories $Y \xleftarrow{f} X \xrightarrow{g} Z$ where $X$, $Y$, and $Z$ are good, and let $W$ denote the cocomma category $\cocomma{f}{g}$.
We then obtain the following cube-shaped diagram with lax top and bottom faces:
\[
\begin{tikzcd}
  \fun(W, B)
  \ar[rr]
  \ar[dd, "(-)^*_W p"']
  \ar[dr]
&&
  \fun(Z, B)
  \ar[dd, "(-)_Z^* p" near end]
  \ar[dr]
\\&
  \fun(Y, B)
  \ar[rr,crossing over]
  \ar[ur, Rightarrow]
&&
  \fun(X, B)
  \ar[dd, "(-)_X^* p"]
\\[1em]
  \cocartfib_p(W)
  \ar[rr]
  \ar[dr]
&&
  \cocartfib_p(Z)
  \ar[dr]
\\&
  \cocartfib_p(Y)
  \ar[from=uu, crossing over, "(-)_Y^* p"' near start]
  \ar[rr]
  \ar[ur, Rightarrow]
&&
  \cocartfib_p(X) \rlap{.}
\end{tikzcd}
\]
The top square is a comma square since $W$ is a cocomma.
The bottom square is a comma square by \cref{descent-cocomma}.
The three vertical maps $(-)^*_X p$, $(-)^*_Y p$, and $(-)^*_Z p$ are invertible by assumption.
The fourth vertical map $(-)^*_W p$ is thus invertible by uniqueness of comma categories.
Thus good categories are closed under cocomma.

The facts that good categories are closed under localisation and sequential colimits similarly follow from \cref{descent-localisation,descent-sequential}.
\end{proof}

\begin{remark}
The problem of defining a fully coherent (as opposed to wild) unstraightening functor is in some sense orthogonal to the problem of showing that the unstraightening functor is invertible.
Indeed, by the fundamental theorem of category theory, the statement that a functor is invertible involves only its action on objects and morphisms.
So far we have \emph{only} considered $\cocartfib(C)$ as a \emph{wild} category, and so have no choice but to consider $(-)_C^* p \co \fun(C, B) \to \cocartfib(C)$ as a wild functor.
However, given a universal cocartesian fibration $\cat_\bullet \to \cat$, it becomes possible to realise $\cocartfib(C)$ as a fully coherent category and we expect it to then also be possible to realise the straightening functor as a fully coherent functor (\cf \cite[Section 7.7]{cisinski25}).
\end{remark}

\printbibliography

@article{ayala20,
  title = {Fibrations of $\infty$-categories},
  author = {Ayala, David and Francis, John},
  year = 2020,
  month = feb,
  journal = {Higher Structures},
  volume = {4},
  number = {1},
  pages = {168--265},
  doi = {10.21136/HS.2020.05}
}

@misc{cisinski22,
  title = {The universal coCartesian fibration},
  author = {Cisinski, Denis-Charles and Nguyen, Hoang Kim},
  year = 2022,
  month = oct,
  number = {arXiv:2210.08945},
  eprint = {2210.08945},
  primaryclass = {math},
  publisher = {arXiv},
  doi = {10.48550/arXiv.2210.08945},
  archiveprefix = {arXiv}
}

@book{cisinski25,
  title = {Synthetic Category Theory},
  author = {Cisinski, Denis-Charles and Cnossen, Bastiaan and Nguyen, Kim and Walde, Tashi},
  year = 2026
}

@misc{gratzer26,
  title = {The $\infty$-category of $\infty$-categories in simplicial type theory},
  author = {Gratzer, Daniel and Weinberger, Jonathan and Buchholtz, Ulrik},
  year = 2026,
  month = feb,
  number = {arXiv:2602.02218},
  eprint = {2602.02218},
  primaryclass = {cs},
  publisher = {arXiv},
  doi = {10.48550/arXiv.2602.02218},
  archiveprefix = {arXiv}
}

@article{hebestreit25,
  title = {A short proof of the straightening theorem},
  author = {Hebestreit, Fabian and Heuts, Gijs and Ruit, Jaco},
  year = 2025,
  month = may,
  journal = {Transactions of the American Mathematical Society, Series B},
  volume = {12},
  number = {19},
  pages = {697--747},
  issn = {2330-0000},
  doi = {10.1090/btran/225},
  copyright = {https://creativecommons.org/licenses/by/3.0/}
}

@article{kelly80,
  title = {A unified treatment of transfinite constructions for free algebras, free monoids, colimits, associated sheaves, and so on},
  author = {Kelly, G.M.},
  year = 1980,
  month = aug,
  journal = {Bulletin of the Australian Mathematical Society},
  volume = {22},
  number = {1},
  pages = {1--83},
  issn = {0004-9727, 1755-1633},
  doi = {10.1017/S0004972700006353},
  copyright = {https://www.cambridge.org/core/terms}
}

@book{lurie09,
  title = {Higher Topos Theory},
  author = {Lurie, Jacob},
  year = 2009,
  series = {Annals of Mathematics Studies},
  publisher = {Princeton University Press}
}

@misc{martini22,
  title = {Cocartesian fibrations and straightening internal to an $\infty$-topos},
  author = {Martini, Louis},
  year = 2022,
  month = may,
  number = {arXiv:2204.00295},
  eprint = {2204.00295},
  primaryclass = {math},
  publisher = {arXiv},
  doi = {10.48550/arXiv.2204.00295},
  archiveprefix = {arXiv}
}

@misc{rijke17,
  title = {The join construction},
  author = {Rijke, Egbert},
  year = 2017,
  month = jan,
  number = {arXiv:1701.07538},
  eprint = {1701.07538},
  primaryclass = {math},
  publisher = {arXiv},
  doi = {10.48550/arXiv.1701.07538},
  archiveprefix = {arXiv}
}

@misc{uemura25,
  title = {Colimits in the $\infty$-category of $\infty$-topoi and \'etale morphisms},
  author = {Uemura, Taichi},
  year = 2025,
  month = jun,
  number = {arXiv:2506.10431},
  eprint = {2506.10431},
  primaryclass = {math},
  publisher = {arXiv},
  doi = {10.48550/arXiv.2506.10431},
  archiveprefix = {arXiv}
}
\end{document}